\documentclass[11pt,reqno]{amsart}
\usepackage{color}
\usepackage{prettyref}
\usepackage{amstext}
\usepackage{amsthm}
\usepackage{amssymb}
\usepackage[unicode=true,pdfusetitle,
 bookmarks=true,bookmarksnumbered=false,bookmarksopen=false,
 breaklinks=false,pdfborder={0 0 1},backref=false,colorlinks=false]
 {hyperref}

\makeatletter
\numberwithin{equation}{section}
\numberwithin{figure}{section}

\theoremstyle{plain}
\newtheorem{thm}{\protect\theoremname}[section]
\newtheorem{cor}[thm]{Corollary}
\newtheorem{lem}[thm]{\protect\lemmaname}
\newtheorem{conjecture}{\protect\conjecturename}
\newtheorem{question}[conjecture]{Question}

\theoremstyle{definition}
\newtheorem{defn}[thm]{\protect\definitionname}
\newtheorem{hyp}{Hypothesis}
\theoremstyle{remark}

\usepackage{mathrsfs,amsthm,amssymb,bbm,fullpage}

\renewcommand{\limsup}{\varlimsup}

\newcommand{\cC}{\mathcal{C}}

\newcommand{\cE}{\mathcal{E}}

\newcommand{\cS}{\mathcal{S}}
\newcommand{\R}{\mathbb{R}}

\newcommand{\prob}{\mathbb{P}}
\newcommand{\E}{\mathbb{E}}
\newcommand{\eps}{\epsilon}

\newcommand{\eqdist}{\stackrel{(d)}{=}}

\newcommand{\abs}[1]{\lvert#1\rvert}
\newcommand{\norm}[1]{\lvert\lvert#1\rvert\rvert}

\newrefformat{cor}{Corollary \ref{#1}}
\usepackage[normalem]{ulem}

\newcommand{\bfn}{\mathbf{n}}

\makeatother

\providecommand{\conjecturename}{Conjecture}
\providecommand{\definitionname}{Definition}
\providecommand{\lemmaname}{Lemma}
\providecommand{\theoremname}{Theorem}

\begin{document}
\title{Shattering versus metastability in spin glasses}

\author{G\'erard Ben Arous}
\address{G.\ Ben Arous\hfill\break
Courant Institute of Mathematical Sciences\\ New York University\\
New York, NY, USA.}
\email{benarous@cims.nyu.edu}

\author{Aukosh Jagannath}
\address{A.\ Jagannath\hfill\break
Department of Statistics and Actuarial Science \\ 
Department of Applied Mathematics\\
University of Waterloo\\
Waterloo, ON, Canada.} 
\email{a.jagannath@uwaterloo.ca}

\begin{abstract}
Our goal in this work is to better understand the relationship between replica symmetry breaking, shattering, and metastability.
To this end, we study the static and dynamic behaviour of spherical pure $p$-spin glasses above the replica symmetry breaking temperature $T_{s}$.  In this regime, we find that there are at least two distinct temperatures related to non-trivial behaviour. First we prove that there is a regime of temperatures in which the spherical $p$-spin model exhibits a shattering phase. Our results holds in a regime above but near  $T_s$.
We then find that metastable states exist up to an even higher temperature $T_{BBM}$ as predicted by Barrat--Burioni--M\'ezard
which is expected to be higher than the phase boundary for the shattering phase $T_d <T_{BBM}$. We develop this work by first developing a Thouless--Anderson--Palmer decomposition which builds on the work of Subag. 
We then present a series of questions and conjectures regarding the sharp phase boundaries for shattering
and slow mixing.
\end{abstract}

\maketitle

\section{Introduction}
We study here the static and dynamic behaviour of spherical pure $p$-spin glasses in a range of temperatures above the replica symmetry breaking temperature, $T_s$. The understanding of the statics in the glassy phase below this temperature is now quite complete using the classical Parisi approach \cite{MPV87} via a variational formula for the free energy, which in this setting is given by the Crisanti--Sommers formula \cite{CriSom92}. This approach shows that, below $T_s$, the model exhibits what is called ``one-step replica symmetry breaking'' (1 RSB). This approach was made rigorous in the mathematics literature by Talagrand \cite{TalSphPF06}  for even $p$-spin models  and by Chen \cite{ChenSph13} for all $p$ building on the work of Guerra \cite{Guerra2003} and on the works of Aizenman--Sims--Starr \cite{aizenman2003extended} and Panchenko \cite{PanchPF14} respectively. This description can also be completed, and made more geometric, using the more recent understanding  of the topological complexity of the energy landscape \cite{ABC13,ABA13}.  Indeed Subag proved \cite{SubGibbs16} that the Gibbs measure concentrates on bands around the deepest minima and that the free energy (and Gibbs mass) of these bands is equivalent to that of the total system at sufficiently low temperature (see also \cite{arous2020geometry} for an application of this approach to the problem of chaos in temperature).

We concentrate here on a geometric description of the free energy landscape in a range of temperatures \emph{above} this static transition. It is well-known
in the physics literature that another transition occurs at a higher temperature, usually called the dynamical temperature, $T_d >T_s$, and which we refer to here as $T_{sh}$. This temperature was initially introduced by Kirkpatrick and Thirumalai \cite{KT87} as the temperature below which the Langevin dynamics are slow when started from a random point.

This dynamical temperature also has an interesting, purely static interpretation, as the onset of the ``shattering phase", where the free energy landscape is shattered: in this regime,  the free energy of the system is equivalent to that given by a union of an exponentially large number of such bands, whose free energies (and Gibbs masses) are all exponentially small. This is in direct contrast to the aforementioned 1RSB phase. We prove here that the shattering phase exists in an interval $(T_{s},T_{0})$ where $T_0 \leq T_{sh}$ and conjecture that $T_0=T_{sh}$. These bands are disjoint and centred on critical points of the energy.  Our approach to proving shattering of the free energy landscape thus naturally connects to the topological complexity of this landscape and begins by a computation of the TAP free energy using this complexity which extends the work of Subag \cite{SubGibbs16} to a broader range of temperatures.

The shattering phase emerged long ago in a series of work in the physics literature by Kirkpatrick--Thirumalai \cite{KT87}, Kurchan--Parisi--Virasoro \cite{KurParVir93}, and Barrat--Burioni--M\'ezard \cite{barrat1996dynamics}. It was later studied in great depth for many important problems related to sparse, mean-field models of spin glasses and central questions from Theoretical Computer Science and combinatorics, such as random constraint satisfaction and combinatorial optimization problems. See \cite{dembo2013factor, COP19}, \cite{MPZ02,KMRTSZ07,ACRT11,DSS15}, and \cite{achlioptas2008algorithmic,DSS16,sly2016reconstruction}  respectively for a necessarily small sample of such works. In this work, we return to the  dense case of spherical $p$-spin models following the early and fundamental paper by Barrat, Burioni, and M\'ezard \cite{barrat1996dynamics}.

We then turn to dynamical questions. In \cite{barrat1996dynamics}, Barrat--Burioni--M\'ezard introduced another important temperature $T_{BBM}> T_{sh}$ related to metastability of the Langevin dynamics. We show here that below this temperature, the spectral gap is exponentially small, and moreover that there are are exponentially many bands centred on critical points from which the Langevin dynamics takes an exponentially long time to escape. We conjecture that this temperature is the right threshold and that above it, the spectral gap is bounded below with high probability, and relate this conjecture to our recent work on the spectral gap \cite{BAJag17} using the so-called ``two-replica" potential. To study the connection between the free energy landscape, spectral gaps, and exit times, we use here the free energy landscape approach developed by Gheissari and the authors in several works \cite{GJ16,BAJag17,BGJ18b}.

\section{Main results and discussion}

Our goal in this work is to better understand the relationship
between replica symmetry breaking, shattering, and metastability in
spin glasses through the lens of the complexity of the free energy landscape. 
To do so, we focus on a simple class of models, 
the spherical $p$-spin glass models, where one can provide a geometric
perspective on these questions. These models are defined as follows.

Let $\cS_{N}=\{ x\in\R^{N}:\norm{x}_{2}=\sqrt{N}\} $
and for $p\geq 1$ consider the \emph{$p$-spin Hamiltonian}, $H_{N,p}:\cS_{N}\to\R$,
which is given by 
\[
H_{N,p}(x)=\frac{1}{N^{\frac{p-1}{2}}}\sum_{i_{1}\cdots i_{p}=1}^NJ_{i_{1}\cdots i_{p}}x_{i_{1}}\cdots x_{i_{p}},
\]
where $J_{i_{1}\cdots i_{p}}$ are i.i.d. standard Gaussians. We will also be interested in \emph{the Langevin dynamics at temperature $T>0$} for the $p$-spin Hamiltonian, namely the solution to 
the stochastic differential equation
\[
\begin{cases}
dX_t = dB_t - \frac{1}{T} \nabla H_{N,p}(X_t) dt \\
X_0 = x,
\end{cases}
\]
where here $B_t$ is spherical Brownian motion and $\nabla$ is the usual covariant derivative.
Let $Q_x$ denote  the law of  $X_t$ when started from $X_0=x$ and let 
$L$ denote the infinitesimal generator for $X_t$, namely
\[
L=\frac{1}{2} \Delta -\frac{1}{T} \langle\nabla H_{N,p},\nabla \cdot\rangle.
\] 
Here $\langle\cdot,\cdot\rangle$ denotes the induced metric on $\cS_N$.
Recall that $X_t$ is reversible with invariant measure given by the Gibbs measure,
$\pi_T(dx)\propto\exp(-\frac{1}{T}H(x))dx$, where here and in the following $dx$ refers to the uniform measure
on $\cS_N$.

We begin our analysis by developing a Thouless--Anderson--Palmer decomposition
for the free energy landscape of spherical $p$-spin models.
We then use this decomposition to prove a Barrat--Burioni--M\'ezard-type lower bound,
a lower bound for the free energy of the total system that was first developed non-rigorously
in \cite{barrat1996dynamics}. We then turn to understanding
the implications of these results for the shattering of the free energy landscape and metastability
of the corresponding dynamics. We end this section with an extended discussion of the relationship between these two concepts.

\subsection{A Thouless--Anderson--Palmer decomposition}
The starting point of our analysis is to develop a Thouless--Anderson--Palmer
type decomposition for the free energy landscape. Before we can state
this result  we need to first recall the following notions regarding
free energies, complexities, and mixed $p$-spin glass models.

For a Borel set $A\subseteq\cS_{N}$
and an inverse temperature $\beta=T^{-1}>0$, let the \emph{restricted
free energy} be
\[
F_{N}(A;\beta)=\frac{1}{N}\log Z_{N,\beta}(A)=\frac{1}{N}\log\int_{A}e^{-\beta H_{N}(x)}dx,
\]
where $dx$ is the uniform measure on
$\cS_{N}$. The total free energy at inverse temperature $\beta=T^{-1}>0$
is then $F_{N}(\cS_{N};\beta)=F_{N}(\beta).$ For
$x,y\in\cS_{N}$, define their \emph{overlap} as $R(x,y)=(x\cdot y)/N,$ where $ \cdot$ 
denotes the usual Euclidean inner product
and for a point $x$, let $B(x,q,\eta)$ denote the band
\[
B(x,q,\eta)=\left\{ y\in\cS_{N}:\abs{R(x,y)-q}\leq\eta\right\}.
\]

Now recall the following results regarding
the complexity of spherical spin glasses. 
For a Borel set $A\subseteq\cS_N$,
let 
\[
\cC_{N}(A)=\left\{ x\in\cS_{N}:\frac{1}{N}H_{N,p}(x)\in A,\nabla H_{N,p}=0\right\} 
\]
denote the collection of critical points of $H_{N}$ with (normalized)
energy below $E$ and let $|\cC_{N}(A)|$ denote the cardinality of this set.
Recall the \emph{complexity}
$\Theta_{N}:\R\to\R$,
\[
\Theta_{N}(E)=\frac{1}{N}\log|\cC_{N}((-\infty,E))|,
\]
and the \emph{asymptotic complexity} from \cite{ABC13}, 
\[
\Theta(E)=\begin{cases}
\frac{1}{2}\log(p-1)-\frac{p-2}{4(p-1)}E^{2}-\frac{2}{E_{\infty}^{2}}\int_{E}^{E_{\infty}}(z^{2}-E_{\infty}^{2})^{1/2}dz & E\leq E_{\infty}\\
\frac{1}{2} \log(p-1) - \frac{p-2}{4(p-1)} E^2 & E_{\infty}\leq E\leq 0,\\
\frac{1}{2}\log(p-1), & u\geq 0.
\end{cases}
\]
where $E_{\infty}=-2\sqrt{(p-1)/p}.$ Let $E_{0}$ denote the
zero of $\Theta(E)$. Note that $E_{0}<E_{\infty}$ since the
asymptotic complexity is strictly increasing for $E\leq E_{\infty}$.
We will mainly be concerned with the first regime, $E\leq E_\infty$, in this paper.
To understand the importance of the asymptotic complexity, recall
that by combining the asymptotic complexity calculation of Auffinger, \v{C}erny, and one of us \cite{ABC13}
with the second moment computation of Subag \cite{subag2017complexity},
we have that 
$\Theta_{N}(E)\to\Theta(E)$
in probability for each $E_{0}\leq E\leq E_{\infty}$. (In fact, those results
provide much sharper convergence results, some of which we will use
in the following.) See also the work of Auffinger--Gold \cite{auffinger2020number} for more refined information on the topological complexity of the landscape.

We now need to introduce a one-parameter family of mixed $p$-spin models whose importance,
to our knowledge, was first observed by Subag in \cite{SubGibbs16} and we call here the \emph{co-dimension 1 models}.  
Let $0\leq q\leq 1$ and let $\tilde{H}_{q,N}(x):\cS_{N-1}\to\R$ be given by
\begin{equation}
\tilde{H}_{q,N}(x)=\sum_{k=2}^{p}\alpha_{k}(q)\sqrt{\frac{N}{N-1}}H_{N-1,k}(x)\label{eq:h-tilde}
\end{equation}
where $\alpha_{k}(q)=\sqrt{{p \choose k}(1-q^{2})^{k}}q^{p-k}$. 
(Note that the dimension of the sphere here has changed.)
The co-dimension 1 models are, effectively, the model restricted to a co-dimension 1 sphere which has latitude $q$ with respect to a critical point. For a more precise statement see  \cite[Sec. 4]{SubGibbs16} or Section~\ref{sec:Free-energy-of-rings} below.
We denote the limiting free energy of these models by 
\begin{equation}
F_2(q,\beta) = \lim \frac{1}{N}\log \int_{\cS_{N}} e^{-\beta \tilde{H}_{q,N+1}(x)}dx.
\end{equation}
The almost sure existence of this limit is a consequence of the Crisanti--Sommers formula for general
mixed $p$-spin models on the sphere developed by Chen \cite{ChenSph13}. Finally,
let 
\begin{equation}\label{eq:rate-func-spherical}
I(x) = - \frac{1}{2}\log(1-x^2).
\end{equation} 
With this in hand, we define the \emph{Thouless--Anderson--Palmer free energy} to be 
\begin{equation}
F_{TAP}(E,q,\beta) = -\beta q^p E+F_2(q,\beta) - I(q).
\end{equation}
We note here that $F_{TAP}$ is an extension of what is usually called the TAP free energy in the physics literature
to a broader range of overlaps $q$. (We explain this connection momentarily.)

We are now in the position to develop a TAP decomposition for the free energy.
In the following, for two sequences of random variables $(X_{N})$ and $(Y_{N})$,
we say that $X_{N}\geq Y_{N}+o_{\mathbb{P}}(1)$ if $X_{N}\geq Y_{N}+W_{N}$
for some sequence of random variables $(W_{N})$ with $W_{N}\to0$ in
probability. 
\begin{thm}
\label{thm:ring-FE-lower-bound} 
Let $p\geq 4$. For any $r>0$, there is a $\delta(r)$ such that for any 
$E\in(E_0,E_0+\delta)$, any $\sqrt{(1+r)/2}<q<1$, and any $\beta>0$, there are sequences
$\epsilon_{N},\eta_{N}\to0$ and a sequence of (random) sets $A_N\subseteq\cC_N(E-\eps_N,E+\eps_N)$
with:
\begin{align*}
\frac{1}{N}\log\abs{A_N} &= \Theta(E)+o_\mathbb{P}(1)\\
F_N\left(\cup_{x\in A_N}B(x,q,\eta_{N});\beta\right) & = F_{TAP}(E,q,\beta)+\Theta(E)+o_{\mathbb{P}}(1)\\
\sup_{x\in A_N}\abs{F_N(B(x,q,\eta_N);\beta)-F_{TAP}(E,q,\beta)} &= o_{\mathbb{P}}(1),
\end{align*}
and such that the balls $\{B(x,q,\eta_{N})\}_{x\in A_N}$ are
pairwise disjoint and their centres satisfy $R(x,y)<r$ with probability tending to 1.
\end{thm}
This result shows that for energy levels near $E_0$ and any temperature $T>0$,
most of the critical points at that energy level are well-separated, the bands around them 
are disjoint, the free energies of these bands are (asymptotically) the TAP free energy, $F_{TAP}$,
and the free energy of the system restricted to the union these bands is given by the TAP free energy of such a band plus
the complexity of that energy, $\Theta(E)$, i.e., the exponential rate of such bands. We expect this restriction in energy levels 
to be an artifact of our proof technique and that this result holds for all $E_0\leq E\leq E_\infty$.
For more on this see the discussion in \prettyref{sec:discussion} below.
For a discussion of the restriction on $q$ and the case $p=3$ see \prettyref{sec:p=3} below.

With this decomposition in hand, we can now begin to investigate the core questions of this paper,
namely, the relationship between shattering and metastability in spin glasses. Before turning to this discussion
let us briefly pause to comment on the proof of this result. 

Our proof of \prettyref{thm:ring-FE-lower-bound} is inspired by the work of Barrat--Burioni--M\'ezard \cite{barrat1996dynamics} and Subag \cite{subag2017complexity,SubGibbs16}, the latter
following \cite{ABC13}. 
In particular, we extend Subag's analysis to a broader range of temperatures and energies by leveraging
more refined results related to free energies of mixed $p$-spin models and their corresponding variational formulas
recently developed by several authors \cite{TalSphPF06,AuffChen13,JagTob16boundingRSB}.
We note here that the TAP free energy investigated here, $F_{TAP}$, is equivalent to that derived by
Subag in his deep analysis \cite{subag2018free}, though our derivation, following \cite{SubGibbs16} is slightly different.
For the case $p=2$ via the TAP approach see the work of Belius--Kistler \cite{belius2019tap}.

\subsection{The Barrat--Burioni--M\'ezard lower bound}\label{sec:BBM-bound-intro}
In their fundamental study Barrat, Burioni, and M\'ezard \cite{barrat1996dynamics} introduced a lower bound
for the free energy of the total system in terms of the TAP decomposition.
We prove a weaker form of this lower bound as a consequence of the preceding decomposition. 
To state this result, let us recall  the following definitions.

For any $E_{0}\leq E\leq E_{\infty}$,
let $\beta_{*}(E)$ be the smallest $\beta$ such that there is a
strictly positive solution to the equation 
\begin{equation}
(1-q^{2})q^{p-2}=\frac{1}{2\beta(p-1)}\left(-E-\sqrt{E^{2}-E_{\infty}^{2}}\right)\label{eq:fixed-point}
\end{equation}
and for any $\beta\geq \beta_*(E)$, let $q_*(E,\beta)$ denote the corresponding solution,
and for $\beta\geq \beta_{*}(E_\infty)$, let $q_{**}(\beta) =q_*(E_\infty,\beta)$.  Observe 
that $\beta_*(E)$ is decreasing in $E$, so that $\beta_*(E_0)\leq \beta_*(E)$ for all $E\geq E_0$.
We define the \emph{Barrat--Burioni--M\'ezard} (BBM) temperature to be 
\[
T_{BBM} = \beta_*(E_0)^{-1}.
\]
Define the  \emph{replica symmetric Thouless--Anderson--Palmer (TAP) free energy}:
 \begin{equation}
F_{RS}(E,q,\beta)=-\beta q^{p}E-I(q)+\frac{1}{2}\left\{\beta^{2}\left(1-q^{2p}-pq^{2p-2}(1-q^{2})\right)\right\} .\label{eq:FRS}
\end{equation}
We note that in the physics literature it is more common to refer to  $F_{RS}$ as the TAP free energy see, e.g., \cite{KurParVir93,barrat1996dynamics,CastCav05}.
As we shall show in \prettyref{cor:ftap-is-frs} below,  for $T<\beta_*^{-1}(E)$ and $q\geq q_*(E,\beta)$ we have that $F_{TAP}(E,q,\beta)=F_{RS}(E,q,\beta).$

For $T\leq T_{BBM}$, let
\[
\cE_T = \{ E \in [E_0,E_\infty]: T<\beta_*(E)^{-1}\},
\]
and define
\begin{equation}\label{eq:BBM-func}
F_{BBM}(\beta) = \max_{\substack{E\in\cE_T \\ q\in [q_*(E,\beta),1]}} F_{RS}(E,q,\beta)+\Theta(E).
\end{equation}
Let us also define the following modification
of $F_{BBM}$.
Let $\cE_{T,r}=\cE_T\cap\{E\in[E_0,E_0+\delta_0(r)]: q_*(E,\beta)>\sqrt{(1+r)/2}\}$ and 
\[
U(\beta) = \sup_{r>0} \sup_{\substack{E\in\cE_{T,r} \\ q\in [q_*(E,\beta),1]}} F_{RS}(E,q,\beta)+\Theta(E),
\]
where here $\delta_0$ is as in \prettyref{thm:shattering}.

It was predicted in \cite{barrat1996dynamics}, that $F_{BBM}(\beta)$ is a lower bound for the total free energy, $F(\beta),$
for all temperatures below $T_{BBM}$. 
As an immediate consequence of \prettyref{thm:ring-FE-lower-bound}, we obtain the following.
 \begin{cor}[Barrat--Burioni--M\'ezard lower bound]
For $p\geq 4$
and $T<T_{BBM}$ we have
\begin{equation}\label{eq:BBM}
F(\beta)\geq U(\beta).
\end{equation}
\end{cor}
This result show us that the TAP free energy plus the corresponding complexity, when restricted to energy levels near $E_0$ and overlaps $q\geq q_*(E,\beta)$, is a lower bound for the free energy at all temperatures. 
At this point, the curious reader will of course ask if this lower bound is tight. 
In their work, Barrat--Burioni--M\'ezard also predicted that below a different temperature
this lower bound is tight. This is related to the phenomenon of shattering
which we discuss now.

\subsection{The shattering phase}
One of our main results is the proof of the existence
of a shattering phase in spherical $p$-spin models. 
To state this result, let us begin by first providing a precise
notion of shattering and recalling the notion of replica symmetry breaking.

A band $B(x,q,\eta)$ is \emph{$c$-subdominant} for some $c>0$ if
\[
\pi_T(B(x,q,\eta))\leq \exp(-cN).
\]
Note that in terms of free energies this can be equivalently written as 
\[
F_N(\beta)-F_N(B(x,q,\eta);\beta)>c.
\]
We can now define the notion of shattering. We
tailor our definition to the precise form of shattering that occurs
here. 
\begin{defn}
\label{def:shattering} For fixed $T>0, E\in\R, r\geq 0,$ and $0<q<1$, we say the free
energy landscape is $(E,q,r)$-\emph{shattered} at temperature $T$
if there are $c,c'>0$ such that for some sequence
$\epsilon_N,\eta_N,\delta_N\to0$ we have that the following occurs with probability tending to $1$:
there is a sequence of sets $A_{N}\subseteq\cC_{N}([-E-\epsilon_N,-E+\epsilon_N])$,
such that for $\beta = T^{-1}$,
\begin{enumerate}
\item (positive complexity) $\frac{1}{N}\log\abs{A_{N}}\geq c$,
\item (separation) for all distinct $x,y\in A$, we have that $B(x,q,\eta_N)\cap B(y,q,\eta_N)=\emptyset$
and that $R(x,y)<r$,
\item (sub-dominance) and for each $x\in A$, the band $B(x,q,\eta_N)$ is $c'$-subdominant,
\[
F_N(\beta)-F_N(B(x,q,\eta_N);\beta)>c'>0.
\]
\item (free energy equivalence) Furthermore,  we have that
\[
F_{N}(\beta)-F_{N}(\cup_{x\in A_N}B(x,q,\eta_N),\beta)\to0
\]
in probability.
\end{enumerate}
\end{defn}
Informally, shattering occurs when there are exponentially many regions which are well-separated and 
whose combined free energy is equivalent to that of the total
system, but each of which has exponentially small mass with respect to the Gibbs measure. In this setting, we will choose these regions to be bands around critical points of a certain energy. We note here, however, that shattering of the free energy
landscape does not imply that the Gibbs measure is supported only on those bands. For more on this see \prettyref{sec:discussion} below. We also note here that the notion of shattering used by Aclioptas--Coja-Oglan \cite{achlioptas2008algorithmic} (at zero temperature) is slightly stronger in that it has an additional condition on the change in energy along paths from one band to another.\footnote{A similar statement can be shown in this setting and is implicit in our study of metastability however as it is not necessary for our discussion we do not include this here.}

Next we turn to replica symmetry breaking. We do not provide
a detailed description of the replica symmetry breaking picture here.
For this we refer the reader to the texts  \cite{MPV87,MM09,PanchSKBook}
and the many deep works in recent years on the rigorous understanding of the replica symmetry breaking phase of mean field spin glass models,
see, e.g., \cite{BoltSznit98,Arg08,PanchUlt13,AuffChen13,jagannath2017approximate,SubGibbs16,AufJag16}
for a necessarily small selection. Let us instead recall the simpler, analytical characterization
of replica symmetry which suffices for our purposes:
the $p$-spin model is said to be in the \emph{replica symmetric} phase
if the limit of the total free energy, $F(\beta)$, satisfies $F(\beta)=\frac{\beta^{2}}{2}$,
and otherwise it is said to be in the \emph{replica symmetry breaking}
phase. 
Let $T_{s}$ be is the phase
boundary for the replica symmetric phase, namely
\begin{align*}
T_{s} & =\max\left\{ T>0:F(T^{-1})=\frac{1}{2T^{2}}\right\} .
\end{align*}
That $T_{s}$ is positive and finite was shown in \cite{TalSphPF06}. We then
define $T_{sh}$ to be 
\[
T_{sh}=\sqrt{p\frac{(p-2)^{p-2}}{(p-1)^{p-1}}}.
\]
We note that in the physics literature $T_{sh}$ is more commonly called $T_{d}$ 
or the ``dynamical replica symmetry breaking phase transition"  \cite{BCKM98,CastCav05,MM09}.
(We discuss this in more detail in Section~\ref{sec:discussion} below.)
Note that, as a consequence of our analysis,  $T_{s}<T_{sh}<T_{BBM}$. 
For the reader's convenience we include an alternative, direct proof in Appendix \ref{sec:teperatures-not-equal}.

We now turn to our main results regarding shattering. 
Evidently from \eqref{eq:BBM}, shattering will occur
if that bound is tight and a maximizing energy $E$ has positive complexity. 
We prove this by an explicit computation.
\begin{thm}
\label{thm:FRS=FTAP} For every $p\geq 4$ here is a $T_{s}< T_{0}\leq T_{sh}$
such that for all $T\in(T_{s},T_{0})$ we have that for $\beta=T^{-1}$,
\begin{align}
F(\beta)= U(\beta) =\beta^2/2.\label{eq:FRS-FTAP}
\end{align}
Furthermore for such $T$, the maximum in \eqref{eq:BBM-func} is achieved at a pair $(E,q)=(E(\beta),q(\beta))$
with $q=q_*(E,\beta)$ and $E_0<E\leq E_\infty$ which satisfies
$E = - \beta(q^p +p(1-q^2)q^{p-2}),$
and such that the map $\beta\mapsto (E(\beta),q(\beta))$ is continuous and has $E(\beta)\to E_0$ as $\beta\to\beta_{s}$.
\end{thm}
With this in hand we see that if we let $(E(\beta),q(\beta))$ be any optimal pair from in \prettyref{thm:FRS=FTAP}
we obtain the following.
\begin{thm}
\label{thm:shattering} For every $p\geq4$, there is a $T_{0}>0$
with $T_{s}< T_{0}\leq T_{sh}$  and an $r>0$ such that the free energy landscape is
$(E(\beta),q(\beta),r)$-shattered with probability tending to 1 for all $T_{s}<T\leq T_{0}$. 
\end{thm}
\noindent This result shows that at moderate temperatures, the $p$-spin
model is replica symmetric but the free energy landscape is shattered.
In particular, it verifies the existence of the shattering phase that as predicted
in the physics literature.
Our restriction to $T$ near $T_s$ is related to the disjointness issue discussed after 
\prettyref{thm:ring-FE-lower-bound}. In particular, 
it is expected that $T_{sh}$ is in fact the sharp phase boundary; see Conjecture \ref{conj:shattering} and the surrounding discussion. We discuss this in more detail momentarily, however, let us first
examine the implications of this approach to metastability in spin glasses.

\subsection{Metastability}
The computation of the TAP free energy also has important consequences for the dynamics of spin glasses. 
In particular, we observe that $T_{sh}$ is not the onset of metastability. Instead, we find that metastable states exist up until $T_{BBM}$. We will study metastability from two standpoints: exit times from sub-dominant sets and spectral gaps.
To state our results we need the following definitions. 

For a set $E$, we let $\pi_T(dx\vert E)$ denote the Gibbs measure conditioned on $E$.
For a point $x$
and a pair, $(q,\eta)$, let $\mathcal{Q}_N(x,q,\eta,T,u)$ denote the probability that Langevin dynamics
at temperature $T$ exits the band $B(x,q,\eta)$ before time $\exp(N u)$
when started within that band,
\[
\mathcal{Q}_N(x,q,\eta,u) = \int Q_y(\tau_{B(x,q,\eta)^c}\leq e^{N u}) \pi_T(dy\vert B(x,q,\eta)).
\]
Finally,  
let $0=\lambda_0(\beta)\leq \lambda_1(\beta)\leq ....$ denote the ranked eigenvalues of $-L$.
We then have the following.

\begin{thm}\label{thm:metastability-main}
For every $p\geq 4$ and $T<T_{BBM}$ there are $E_0<E<E_\infty$, $0<q<1,$
$h,c,C,\eta_0>0$, and a sequence $\epsilon_N\to0$ such that for $\eta<\eta_0$ the following holds  with probability tending to 1: 
there is an $A_{N}\subseteq\cC_{N}(E-\epsilon_N,E+\epsilon_N)$ with $\frac{1}{N}\log |A_N|\geq C $,
such that for every $x\in A_{N}$, we have that the band $B(x,q,\eta)$ is $c$-subdominant and for any $0\leq\theta<1$,
\[
\sup_{x\in A_N} \mathcal{Q}_N(x,q,\eta,\theta h)\leq e^{-N(1-\theta)h}.
\]
Furthermore, for such $T$, there are $C',c'>0$ such that 
\[
\prob(-c' <\frac{1}{N} \log \lambda_1(T) \leq -C') \to 1.
\]
\end{thm}
\noindent The first part of this result shows that for $T<T_{BBM}$, there are exponentially many bands $\{B(x,q,\eta_N)\}_{x\in A_N}$ 
which all have exponentially small Gibbs mass, but such that started within any such bands, Langevin dynamics
takes exponential time to escape. The second shows that similarly, the spectral gap is exponentially small with probability tending to 1.  It is interesting to note here that it is the lowest energies that govern metastable behaviour at high temperatures. Indeed,
since the map $E\mapsto \beta_*(E)^{-1}$ is decreasing, $T_{BBM}$ is the highest temperature at which the bands a $q_*$ are well-defined exists and it corresponds to the onset of the existence of these bands at $E_0$. This is to be contrasted with 
the shattering result from \prettyref{thm:shattering} in which the lowest energies only govern the shattering phase near $T_s$.
This is consistent with the work of Barrat--Burioni--M\'ezard \cite{barrat1996dynamics} which argued
that the shattering transition is governed by the bands around the highest energy levels (which are the most numerous but, in a sense, the least stable in temperature) whereas
the metastability transition is governed by the bands around the lowest energy levels (which are the least numerous but 
the most stable in temperature).

\subsection{Discussion: Shattering versus Metastability}\label{sec:discussion}
Let us now turn to the interpretation of these results, specifically a discussion of
what they say about the relation between  the temperatures $T_{s},T_{sh},$ and $T_{BBM}$. 

Theorem \ref{thm:shattering}~shows that spherical $p$-spin models exhibit a shattering phase and, more precisely,
that shattering occurs in a range of temperatures with a lower endpoint that is at least
$T_{s}$. 
This phase is one of the hallmarks of glassy systems and we expect that it is the only regime
where the exponential complexity of the landscape of spin glass models
is clearly felt for the statics of $p$-spin models. 
To our knowledge, this is the first rigorous result regarding a shattering
transition in a spin glass model at positive temperature, though its is important to
note the closely related work at zero-temperature for the solution space geometry for constraint satisfaction
problems, see, e.g., \cite{achlioptas2008algorithmic,ACRT11,sly2016reconstruction}.
It is expected \cite{barrat1996dynamics} that the precise range of temperatures for which
shattering occurs is in fact $(T_{s},T_{sh}]$:
\begin{conjecture}
\label{conj:shattering} For each $T_{s}\leq T\leq T_{sh}$, the free energy landscape
 is $(E,q,r)$-shattered for some $(E,q,r)$. Furthermore, the free energy landscape is not shattered
for any $T>T_{sh}$ or $T<T_s$. 
\end{conjecture}
\noindent  For the first part of the conjecture, we expect that 
the approach we present here should be essentially sufficient. 
In particular, we reduce its  proof to the following hypothesis regarding the geometry of the set of critical points
of the Hamiltonian at a certain energy level. 
\begin{defn}\label{def:separated}
We say that the landscape at energy $E$ is \emph{essentially $r$-separated}
 if for some 
$\delta>0$, we have that for every $\eps>0$ small enough
\[
\limsup_{N\to\infty}P(\abs{\{x,x'\in\cC_N(E-\eps,E+\eps): r<R(x,x')<1\}} \geq e^{-N\delta}\abs{\cC_N(E-\eps,E+\eps)})= 0.
\] 
\end{defn}
This condition says that the number of pairs of distinct  critical points with energy near $E$
is exponentially smaller than the number of critical points in this energy window.
Our results, e.g., \prettyref{thm:ring-FE-lower-bound} and \prettyref{thm:shattering},
are a consequence of essential $r$-separation at energy levels near $E_0$,  see \prettyref{lem:subag-separated} below.
We expect that this condition holds at all energy levels between $E_0$ and $E_\infty$:
\begin{hyp}\label{hyp:1}
For every $E_0\leq E\leq E_\infty$, the landscape at energy $E$ is essentially
$r$-separated for some $0<r<(p-3)/(p-1)+\iota = 2(q_{**}(E_\infty,\beta_{sh})^2)-1+\iota$
and some $\iota$ sufficiently small.
\end{hyp}
We show in \prettyref{sec:conj1} 
below that Conjecture~\ref{conj:shattering}~is implied by Hypothesis~\ref{hyp:1}~for each $p\geq 3$.
Proving Hypothesis~\ref{hyp:1} (and related hypotheses)
would have many implications such as a generalization of the 
BBM bound to all temperatures and the TAP decomposition for all reasonable overlaps.
We leave the interesting question of verifying 
this hypothesis for future work:
\begin{question}
Does Hypothesis~\ref{hyp:1} hold for $p\geq3$?
\end{question}

The second part of Conjecture~\ref{conj:shattering}, however, is far more subtle. In our work, we 
lower bound the combined free energy of the bands around the deepest critical points
whose free energies are asymptotically replica symmetric
in the limit of large $N$ and small $\eta$. In the language of
the physics literature, we  lower bound the free energy of the  deepest (replica symmetric)
Thouless--Anderson--Palmer (TAP) states and show that their combined free energy
is asymptotically lower bounded by their individual free energy plus their ``configurational entropy''. 
In \cite{barrat1996dynamics}, it was predicted that that 
the free energy plus configurational entropy for any of the replica
symmetric TAP states is strictly less than that of the total system.
That being said, it is not clear that these states are the only states
that could induce shattering, e.g., there could be more exotic states which induce
shattering than bands around critical points.  In the other direction, in \cite{barrat1996dynamics},
it was also predicted that for $T<T_s$, the free energy plus configuration entropy for any replica symmetric
with $E>E_0$ is strictly less than that of the total system. That the value at $E=E_0$  (where $\Theta(E_0)=0$)
matches the free energy of the total system for $T<T_{s}$, was show by Subag \cite{subag2018free},  and that for 
sufficiently low temperatures no other energy levels are relevant \cite{SubGibbs16}. 
We leave these important questions for future work.

One might expect that in this regime, the Gibbs measure is shattered as well. 
By this we mean that the Gibbs measure is supported on the bands as in Definition~\ref{def:shattering},
i.e., that the probability of the union of those bands is tending to 1 or is at least order 1 asymptotically. 
This is of course, not equivalent to shattering in the sense of free energies which only guarantees
that on the exponential scale, the Gibbs measure is roughly equivalent to that conditioned on the bands.\footnote{
Indeed, similar questions regarding the difference between the Gibbs measure and a ``free energy" equivalent version  arises in understanding the perturbative approach to computing free energies commonly used in the literature  \cite{TALAGRAND2003477,PanchPF14} to prove free energy formulas. See \cite{BABC08} for a careful study of these issues in a related problem.}
To our knowledege, this problem is not considered in the physics literature. 
We leave this intriguing question to future research (which we state informally):
\begin{question}
For each $T_{s}\leq T\leq T_{sh}$ the Gibbs measure
 is shattered. Furthermore, the Gibbs measure is not shattered
for any $T>T_{sh}$ or $T<T_s$. 
\end{question}

Let us now turn to discussing metastability. 
Theorem~\ref{thm:metastability-main} shows us that metastability occurs (at least) up to a higher temperature $T_{BBM}>T_{sh}$
and that for $T<T_{BBM}$ there are exponentially many metastable states 
in the sense that there are exponentially many bands whose free energies are each less than that of the total system
and such that the exit time of any one of these bands, when started within it, is exponentially small. 
Furthermore, we see that slow mixing occurs in this regime since the spectral gap is exponentially small.
On the other hand, it was shown by Gheissari and one of us \cite{GJ16} that for $T$ sufficiently large, the spectral gap is order 1 (more precisely, it was shown there that $\pi_T$ admits a Logarithmic Sobolev inequality with constant which is bounded away from $0$). In light of this, it is interesting to ask when the onset of slow mixing occurs at the level of spectral gaps. Though it is not clear to us at this time, it seems reasonable to hope that this is precisely $T_{BBM}$. We leave this as another 	exciting open question.
\begin{question}
For $T>T_{BBM}$, do we have that $P(\lambda_{1}(T)>c)\to1$ for some $c>0$?
\end{question}

Let us now compare these results to our recent work in \cite{BAJag17}. There we showed that, for both
the Ising and spherical $p$-spin models, the spectral gap is exponentially
small up to a temperature $T_{2}$ and that $T_{s}<T_{2}$.\footnote{Slow mixing for $T\in(T_{s},T_{2})$ is stated there only for the
Ising spin case, however, it can be easily extended to the spherical
case due to the result of Ko in \cite{ko2020free}. See, e.g., the
recent survey \cite{jagannath2019dynamics} for details.} 
This work followed an alternative approach to that considered
here, namely making rigorous some of the predictions surrounding 
the ``two replica potential'' \cite{KurParVir93}. There $T_{2}$ is the temperature below which there is
a free energy barrier for the overlap distribution (see \cite{BAJag17} for a
precise definition of this). The relationship between $T_{2}$ and
$T_{BBM}$ is not clear at this time, though it seems natural to expect the following.
\begin{question}
Do we have that have that $T_{sh}<T_{2}<T_{BBM}$?
\end{question}
Progress in this direction would be particularly intriguing as it would be an important step toward
uniting the replica theoretical approach with the complexity approach. 
To our knowledge there has been
little to no study of the phase $T_{sh}<T_{BBM}$ in the physics literature
beyond the initial, fundamental work of Barrat--Burioni--M\'ezard
 \cite{barrat1996dynamics} which, to our knowledge, was the first
result to provide a characterization of $T_{BBM}$.
We note here that one could  also define a $T_{k}$
as the temperature below which there is a free energy barrier
for the ``$k$-replica potential", i.e., the large deviations rate function
for the  overlap array $R^k_N = (R(x^\ell,x^{\ell'}))_{\ell,\ell'\in[k]}$,
where $\{x^\ell\}$ are drawn i.i.d. from $\pi_{N,\beta}$. One could then ask the same question of $T_{k}$.
Could it be that $T_{2}<T_{3}<T_{4}<\ldots<T_{BBM}$?
Or perhaps $T_{k}=T_{BBM}$ for some fixed $k$?
Indeed on a phenomenological level,
a similar picture to the latter case occurs 
in the maximum independent set problem 
where it was shown by Rahman--Virag \cite{rahman2017local} following Gamarnik--Sudan \cite{gamarnik2014limits},
that three replicas suffice to saturate an algorithmic threshold while two do not.

Interestingly, $T_{sh}$ is also expected to have an important
dynamical interpretation. Indeed, the shattering phase is
called the \emph{dynamical replica symmetry breaking phase} in the
statistical physics literature and $T_{sh}$ is usually called the critical temperature for the
\emph{dynamical phase transition} \cite{CugKur93,FraPar95,barrat1996dynamics,BCKM98,CastCav05,MM09}
and typically denoted by $T_d$.
Evidently this phase transition is not in terms of ``ergodicity breaking'' in the sense
of slow mixing. The dynamical interpretation of $T_{sh}$ is instead expected to be the onset
for \emph{slow thermalization} from a uniform at random start.  More precisely, the following is 
our attempt at formalizing the  prediction in the physics literature (though the
exact form may not be correct as stated). Let $P_t$ be the Langevin semigroup, i.e.,
$P_t f(x)= \E_{Q_x} f(X_t)$. 
\begin{conjecture}
Started from the uniform measure, $dx$, Langevin dynamics takes exponential
time to reach equilibrium $\pi_{T,N}$ for all $T<T_{sh}$. In
particular, if we let the \emph{thermalization time} for the uniform measure be
\[
\tau_{*}=\inf\{t:\max_{\norm{f}_{\infty}\leq1}\int(P_{t}f-\int fd\pi)^{2}dx\leq\frac{1}{e}\},
\]
then for $T<T_{sh}$ there is some $c>0$ such
that $\tau_{*}\geq e^{cN}$ with probability tending to 1, and for
$T>T_{sh}$ we have $\tau_{*}=O(1).$
\end{conjecture}
\noindent Stated from a computational perspective, we expect that the difference between $T_{sh}$
and $T_{BBM}$ is the related to difference between exponentially slow mixing from randomized as opposed
to ``worst case'' initializations.  Note that an exponential upper bound
on $\tau_*$  follows by proving an exponential bound on the spectral gap which can be shown by a Holley--Stroock type argument \cite{holley1987logarithmic}, see, e.g., \cite{BAJag17,GJ16,Mat00}. Let us also note here that the dynamical
interpretation of $T_{sh}$ is expected to be felt a the level of the Cugliandolo--Kurchan equations \cite{CugKur93,cugliandolo1995weak} which have been 
 developed  by Dembo, Guionnet and one of us \cite{BADG01,BADG06} and analyzed
in various regimes by Dembo-Guionnet--Mazza \cite{DGM07} and Dembo--Subag \cite{DS20}.

We end here by discussing the extension of this work to the Ising
spin setting. We expect that the $p$-spin model with Ising spins
exhibits a similar picture to that described here for any $p\geq3$.
It is not clear to us at this time how to extend the approach
here to the discrete setting though we expect that it will involve an extension of the recent result of
 Chen--Panchenko--Subag \cite{chen2018generalized} of the work of Subag \cite{subag2018free} to the non-multisampleable regime.  On the other hand, we note here the
recent breakthrough of Bauerschmidt--Bodineau \cite{bauerschmidt2017very}
on the related problem of the spectral gap for the Sherrington--Kirkpatrick
model with Ising spins. See also \cite{eldan2020spectral}. We expect
that the dynamical phase transition, at the level of spectral gaps,
is fundamentally different for the SK model. (In fact,
we expect that the phase transition for the  SK model with Ising spins is distinct from any mixed $p$-spin model 
with sufficiently small $p=2$ term with either Ising or spherical spins.)  We end by noting
that it would be very interesting to understand the connections between the preceding discussion and the 
activated dynamics of spin glasses and activated aging which has received a tremendous amount
of attention  \cite{BABG02, BovFag05, BABC08, MatMou15,CerWas17,gayrard2019aging,baity2018activated}.
In a related direction it would be very interesting to understand the relation between $T_{sh}$ and $T_{BBM}$  and Bolthausen-type iteration schemes for Thouless--Anderson--Palmer equations \cite{Bolt14} and their generalizations \cite{montanari2021optimization,alaoui2020optimization}.

\subsection*{Acknowledgements}
The authors underscore their debt to G. Biroli and C. Cammarota for carefully explaining to us
the many predictions in the physics literature. We are very grateful to E. Subag for a thorough reading 
of the first version of this paper which spotted an important error. The authors thanks A. Aggarwal,  R. Gheissari, J. Kurchan, and C. Luccibello for helpful discussions and references. 
A.J. acknowledges the support of the Natural Sciences and Engineering Research Council of Canada (NSERC). Cette recherche a \'et\'e financ\'ee par le Conseil de recherches en sciences naturelles et en g\'enie du Canada (CRSNG),  [RGPIN-2020-04597, DGECR-2020-00199].

\section{Free energies of bands and the co-dimension 1 models\label{sec:Free-energy-of-rings}}
We begin here by studying the free energy of bands around critical
points.  The central observation in \cite{SubGibbs16} is that at sufficiently low temperatures,
the free energy of these bands is given by the (replica symmetric) TAP free energy.
This observation will be central to our work here. In particular, we will need to extend this study to a broader range of temperatures
and allow for the possibility of replica symmetry breaking.  
To this end, in this section we will study the free energy of bands conditionally on criticality.

Let us first begin by recalling the notion of mixed $p$-spin glass models.
It will be useful to let $\xi(t)=t^{p}$.
Observe that we may view the $p$-spin Hamiltonian, $H_{N,p}$, 
as a centred Gaussian process on $\cS_{N}$ with covariance
\[
\E H_{N,p}(x)H_{N,p}(y)=N\xi(R(x,y)).
\]
More generally,  note that if we let $f(t)=\sum_p a_p^2 t^p$ for some sequence $(a_p)$
that satisfies $f(1+\eps)<\infty$, then we can define the corresponding mixed $p$-spin Hamiltonian 
\[
H_N(x) =\sum a_p H_{N,p}(x).
\]
By our assumption on $f$, $H_N(x)$ is well-defined, centred and has covariance $\E H_N(x)H_N(y)=N f(R(x,y)).$
For brevity, we will abuse notation and simply refer to $f$ as the \emph{model} and $H_N$ as
the Hamiltonian with model $f$.

A central role in our analysis will be played by the co-dimension one models \eqref{eq:h-tilde}. 
Let
\begin{equation}
\xi(t,q)=\sum_{k=2}^{p}\alpha_{k}(q)^{2}t^{k}=\left[\left(1-q^{2}\right)t+q^{2}\right]^{p}-q^{2p}-p(1-q^{2})tq^{2p-2}.\label{eq:xi-q}
\end{equation}
Evidently, $\tilde{H}_{q}(x)$ has covariance
$\E\tilde{H}_{q}(x)\tilde{H}_{q}(y)=N\xi(R(x,y),q).$ When it is clear from context we will sometimes denote $\xi_q(t)=\xi(t,q)$.  

Let us now recall the existence of the free energy for mixed $p$-spin models on the sphere. 
It was shown by Talagrand \cite{TalSphPF06} and Chen \cite{ChenSph13}
that  the Crisanti--Sommers formula \cite{CriSom92}  provides an exact
representation for this free energy. For our purposes, however, it suffices to
note the following.  If $H_N(x)$ is a mixed $p$-spin Hamiltonian with 
model $f$, then the free energy corresponding to $H_{N}$, which we will call
the free energy of the model $f$, at inverse temperature $\beta$
exists:
\begin{equation}
F(\beta;f)=\lim_{N\to\infty}\E\frac{1}{N}\log\int e^{-\beta H_{N}(x)}dx.\label{eq:free-energy}
\end{equation}
Furthermore, for any Borel $A\subseteq\cS_{N}$,
the corresponding restricted free energy $F_{N}(A;\beta,\xi)$ concentrates:
there exists constants $C(\beta,\xi)>0$ such that for $N\geq1$ and
$A$, 
\begin{equation}
P\left(\abs{\frac{1}{N}\log\int_{A}e^{-\beta H_{N}(x)}-\frac{1}{N}\E\log\int_{A}e^{-\beta H_{N}(x)}}>\epsilon\right)\leq Ce^{-cN\epsilon^{2}}.\label{eq:concentration-of-free-energies}
\end{equation}
For a proof of this concentration, see, e.g., \cite[Lemma 13]{GJ16}. In the following,
it will be helpful to define the following free energies. We will
let $F_{N}(\beta)$ and $F(\beta)$ denote the total free energy of
the $p$-spin model:
\[
F_{N}(\beta)=\frac{1}{N}\log\int e^{-\beta H_{N,p}(x)}dx\qquad\text{and}\qquad F(\beta)=\lim_{N\to\infty}\E F_{N}(\beta).
\]
We will also let $F_{2,N}(q,\beta)$ denote
the free energies corresponding to the model $\xi(t;q)$ at inverse
temperature $\beta$. Consequently, by the preceeding we see that we have the relation between
\begin{equation}\label{eq:F2-def}
F_{2,N}(q,\beta)=\frac{1}{N}\log\int e^{-\beta\tilde{H}_{q}(x)}dx\qquad\text{and}\qquad F_{2}(q,\beta)=\lim_{N\to\infty}\E F_{2,N}(q,\beta).\end{equation}
Observe that in the above notation $F_2(q,\beta)=F(\beta;\xi_q)$.

Throughout the following, it will be useful to note the following
regularity properties of the function $q\mapsto F_{2,N}(q,\beta)$
and related functions.
\begin{lem}
\label{lem:uniform-lipschitz} For any $\beta>0$ and $p\geq2$, we
have the following:
\begin{enumerate}
\item There are some $K(p),C(p)>0$ such that the map $q\mapsto F_{2,N}(\beta,q)$
is $K$-Lipschitz on $[0,1]$ with probability $1-Ce^{-N/C}$.
\item There is some $K(p)>0$ such that the map $q\mapsto\E F_{2,N}(q,\beta)$
is $K-$Lipschitz on $[0,1]$.
\item There is some $K(p)>0$ such that map $q\mapsto F_{2}(\beta,q)$ is $K$-Lipschitz on $[0,1]$. 
\end{enumerate}
\end{lem}

\begin{proof}
Let us begin with the first point. Recall that by an application by
Borell's inequality and the Dudley entropy bound, one can show that for each
$k$, there is some $K'(k),C(k)>0$ independent of $N$ such that
for $N\geq1$, 
\begin{equation}
P(\max_{x}\abs{H_{k,N}(x)}\geq NK')\leq Ce^{-N/C}.\label{eq:max-controlled}
\end{equation}
(see, e.g., \cite[Lemma 6]{GJ16}). Furthermore, since $dx$ is normalized
we have that
\[
-\beta\frac{\max_{x}\tilde{H}_{q}(x)}{N}\leq F_{2,N}(\beta)\leq-\beta\frac{\min_{x}\tilde{H}_{q}(x)}{N}
\]
so that, by a union bound and \eqref{eq:max-controlled}, we have
that with probability $1-C\exp(-cN)$ for some $C,c>0$, there is
some $K(p,q,\beta)$ such that 
\[
\abs{F_{2,N}(q,\beta)}\leq K(p,q,\beta).
\]
Let's work on this event.

Since $\alpha_{k}(q)$ is $C^{1}$ for $k\geq2$, explicitly differentiating
\eqref{eq:F2-def} yields
\[
\partial_{q}F_{2,N}(q,\beta)=\int \partial_{q}\tilde{H}_{q}(x)d\mu,
\]
where here $\mu$ is the Gibbs measure $\mu(dx)\propto\exp(-\beta \tilde{H}_{q}(y))dx$.
On the aforementioned event, we have
\[
\abs{\partial_{q}\tilde{H}_{q}}\leq C'\max_{k}\abs{H_{N,k}(x)}\leq C''N,
\]
for some $C'',C'>0$. This yields the first point. 

To obtain the second, note that by differentiation 
\[
\abs{\partial_{q}\E F_{2,N}(q,\beta)}\leq C'\beta\E\max_{k}\max_{x}\frac{H_{k,N}(x)}{N}\leq C\beta\sum_{k\leq p}\E\max_{x}\frac{H_{k,N}}{N}\leq C''\beta
\]
for some $C',C''>0$, where in the last line we used the Dudley entropy bound
mentioned above. This yields the second point. To obtain the third,
simply note that by \eqref{eq:F2-def}, $F_{2}(q,\beta)$ is the point-wise
limit of uniformly Lipschitz functions so that it is uniformly Lipschitz.
\end{proof}

In the following, let $\mathbf{n}=\sqrt{N}e_{1}$ and let $P_{E}$
denote the law of the Gaussian process $(H_{N,p}(x))_{x\in\cS_N}$ conditioned on
the event that $\mathbf{n}$
being a critical point with energy $H_{N,p}(\mathbf{n})=NE$. Call this event $A(\mathbf{n},E)$. 
Let us now recall the following useful computation for the law of $H_{N,p}$ conditionally on this event from 
\cite[Sec. 4]{SubGibbs16}.
To this end, consider the change of variables which takes $x\in \cS_N$ and expresses
it in the form  $x = (q,y)$ where $q=R(x,\mathbf n)$ and $y\in \cS_{N-1}$. 
For clarity, we will sometimes write $q=q(x)$ and $y=y(x)$.
Conditionally on $A(\mathbf{n},E)$, the 
law of $(H_{N,p})$ satisfies
\begin{equation}
(H_{N,p}(x))\vert_{A(\mathbf{n},E)}\eqdist(NEq(x)^p+\tilde{H}_{q}(y(x))).\label{eq:conditional-criticality}
\end{equation}
Consequently, we note the following. Recall $I$ from \eqref{eq:rate-func-spherical} and 
observe that $I$ is locally Lipschitz on $(0,1).$

\begin{lem}
For any $0<q<1$,and $\eta>0$ with $\eta<q\wedge1-q$ we have that,
for some $C,c,K>0$,  if $\frac{K}{N}<\epsilon<\epsilon_{0}$ and
$N\geq1$ then
\begin{equation}
\sup_{E_{0}\leq E\leq E_{\infty}}P_{E}\left(\abs{F_{N}(B(\mathbf{n},q,\eta);\beta)-\max_{t\in[q-\eta,q+\eta]}-\beta Eq^{p}-I(q)+\E F_{2,N}(t,\beta)}\geq\epsilon\right)\leq Ce^{-Nc\epsilon^{2}}\label{eq:band-concentration}
\end{equation}
where $F_{2,N}(q,\beta)$ is the free energy corresponding to $\tilde{H_{q}}(x)$
.\textcolor{red}{{} }
\end{lem}

\begin{proof}
Fix $E_{0}\leq E\leq E_{\infty}$. By \eqref{eq:conditional-criticality},
we have that with respect to $P_{E}$,
\begin{align*}
\frac{1}{N}\log\int_{B(\mathbf{n},q,\eta)}e^{-\beta H(y)}dy & \eqdist \frac{1}{N}\log \int_{B(\bfn,q,\eta)} e^{-\beta q(x)^{p}NE+\beta\tilde{H}_{q}(y(x))} dy\\
&= \frac{1}{N}\log\int_{q-\eta}^{q+\eta}\int_{y\in\cS_{N-1}}e^{-\beta q^{p}NE+\beta\tilde{H}_{q}(y)}dy\left(1-q^{2}\right)^{\frac{N-3}{2}}dq + \frac{1}{N}\log c_N,
\end{align*}
where in the second line we have used the co-area formula with respect to the function $q(x)$,
and rescaled the inner integral to be on the sphere $\cS_{N-1}$ as opposed to the level set of $q$, namely
$\mathbb{S}^{N-2}(\sqrt{N(1-q^2)})$. Here $c_N$ is the ratio of the surface area of $\cS_{N-2}$ to that of $\cS_{N-1}$
which, by Stirling's formula, can be seen to be $O(\sqrt{N})$. Thus $(1/N) \log c_N = o(1)$.

Consequently, up to a deterministic, additive $o(1)$ correction the right hand side of the above is equal to 
\[
\frac{1}{N}\log\int_{q-\eta}^{q+\eta}\exp\left\{ N\left[-\beta q^{p}E-\frac{N-3}{N}I(q)+\frac{N-1}{N}F_{2,N-1}(q)\right]\right\} dq.
\]
Since $I(q)$ is locally Lipschitz and, on the event from \eqref{lem:uniform-lipschitz},
$F_{2,N-1}(q)$ is uniformly $K$-Lipschitz, we see that on said event,
the above is equal to 
\begin{align*}
 & \max_{t\in[q-\eta,q+\eta]}\left[-\beta q^{p}E-I(q)+F_{2,N-1}(q,\beta)\right]+O\left(\frac{1}{N}\right).
\end{align*}
If we let $D_{N}$  be a $\frac{1}{N}-$net of $[q-\epsilon,q+\epsilon]$, then on this event,
\[
\max_{t\in(q-\epsilon,q+\epsilon)}-\beta q^{P}E-I(q)+F_{2,N-1}(q,\beta)=\max_{t\in D_{N}}-\beta q^{p}E-I(q)+F_{2,N-1}(t,\beta)+O(\frac{1}{N}).
\]
Furthermore, by a union bound and \eqref{eq:concentration-of-free-energies}
we see that for any $\epsilon>0$ we have that 
\[
\max_{t\in D_{N}}|F_{2,N}(t,\beta)-\E F_{2,N}(t,\beta)|<\epsilon
\]
with probability $1-C'Ne^{-c'N\epsilon^{2}}$ for some $C',c'>0$.
Repeating the $\frac{1}{N}$-net argument and using the uniform Lipschitzness
of $\E F_{2,N}(\cdot,\beta)$ from \eqref{lem:uniform-lipschitz},
we see that 
\[
P_{E}(\abs{F_{N}(B(x_{N},q,\eta))-\max_{t\in(q-\eta,q+\eta)}-\beta q^{p}NE-I(q)+\E F_{2,N-1}(q,\beta)}\geq\epsilon/2)\leq C'e^{-c'N\epsilon^{2}}\vee Ce^{-CN}
\]
where we have used here the lower bound on $\epsilon$. Decreasing
$c'$ yields the result.
\end{proof}
We now focus on a specific choice of $q$ and study the behaviour of $F_2(q,\beta)$.
In the following, we say that a model, $f$, is replica symmetric at $\beta$
if $F(f,\beta)=\frac{\beta^{2}}{2}f(1).$  It is helpful to recall here the following test for replica 
symmetry of Talagrand  \cite[Prop 2.3]{TalSphPF06}: a model $f$ is replica symmetric
a inverse temperature $\beta$  
if and only if the function $g(t) = \beta^2f(t) +\log(1-t)+t$ satisfies
\begin{equation}\label{eq:tal-test}
g(t)\leq 0 \qquad \forall 0\leq t\leq 1
\end{equation}

Recall now  $\beta_{*}(E)$ and observe that it is increasing in $E$. 
Furthermore, an explicit calculation shows that
\[
\beta_{*}(E_{\infty})^{-1}=\sqrt{(p-1)\frac{(p-2)^{p-2}}{p^{p-1}},}
\]
so that 
\begin{equation}\label{eq:temp-relations}
T_{BBM}=\beta_*(E_0)^{-1}\geq\beta_{*}(E)^{-1}\geq\beta_{*}(E_{\infty})^{-1}>T_{sh}
\end{equation}
for all $E\in[E_{0},E_{\infty}]$, where the inequalities are strict except, of course, at the end points $E\in\{E_0,E_\infty\}$. 
Recall $q_{*}(\beta,E)$ and $q_{**}$ from \prettyref{sec:BBM-bound-intro}.
\begin{lem}
\label{lem:RS-for-large-q} Suppose that $p\geq3$. Let $E_{0}\leq E\leq E_{\infty}$,
$\beta>\beta_{*}(E)$, and $q_{*}=q_{*}(\beta,E)$. There is an $\epsilon=\epsilon(E,\beta)>0$
such that for all $q_*-\eps\leq q\leq 1$, $\xi(t,q)$
is replica symmetric at $\beta$,
that is, 
\[
F_{2}(q,\beta)=\frac{\beta^{2}}{2}\xi(1,q)=\frac{1}{2}\beta^{2}\left(1-q^{2p}-pq^{2p-2}(1-q^{2})\right).
\]
Furthermore if $\beta>\beta_{*}(E_{\infty})$, then the same result
holds for all $E_{0}\leq E\leq E_{\infty}$ and all $q_{**}(\beta)\leq q\leq1$. 

\end{lem}

\begin{proof}
We begin with the following observations. 
Observe that the statement that $\xi(t,q)$ is replica symmetric is equivalent
to the statement that $F_2(q,\beta)= F(\beta;\xi_q)=\beta^2 \xi(1;q)$.
Recall that by Talagrand's test \eqref{eq:tal-test} that this holds
provided that the function
\[
f(t)=\beta^{2}\xi(t;q)+\log(1-t)+t,
\]
satisfies $f(t)\leq0$ for all $0\leq t\leq1$. To this end, note that $f(0)=f'(0)=0$
and that
\[
f''(t)=\beta^{2}\xi''(t;q)-\frac{1}{(1-t)^{2}}=\frac{\beta^{2}g(t)-1}{(1-t)^{2}},
\]
where
\[
g(t)=p\cdot(p-1)\cdot(1-q^{2})^{2}(t(1-q^{2})+q^{2})^{p-2}(1-t)^{2}.
\]
For $p\geq4,$ observe that $g$ has three critical points, at 
\[
t_1=1\qquad t_2=-\frac{q^{2}}{1-q^{2}}\qquad t_3=\frac{p(1-q^{2})-2}{p(1-q^{2})},
\]
which satisfy $t_2<t_3<t_1$,
and for $p=3$, $g$ has two critical points at $t_1$ and $t_3$.
Futhermore, we have that
\[
g''(t_3) = -2(\frac{p-2}{p})^{p-3}<0
\]
so that $t_3$ is a local maximum, so that $g$ is decreasing for $t_3\leq t\leq t_1$. 
Notice that
$t_3$ is non-positive provided that $q^{2}\geq\frac{p-2}{p}$. This holds for $q\geq q_*>\sqrt{\frac{p-2}{p}}$ by definition of $q_*$ 
since $\beta>\beta_*(E)$. Putting
these observations together we see that $g(t)\leq g(0)$ for $0\leq t\leq1$. Thus it suffices
to check that for such $q$, we have that $\beta^{2}g(0)\leq1$. This is equivalent
to showing that 
\[
(1-q^{2})q^{p-2}\leq\frac{1}{\beta\sqrt{p\cdot p-1}}.
\]
Let us now turn to proving the claims in turn.

We begin with the first. Since $q_{*}$ solves the fixed
point equation \prettyref{eq:fixed-point} the above is equivalent
to
\[
\frac{1}{2\beta(p-1)}\left(-E-\sqrt{E^{2}-E_{\infty}^{2}}\right)\leq\frac{-E_{\infty}}{2\beta(p-1)}.
\]
which holds for $E\leq E_{\infty}$ since $x-\sqrt{x^{2}-1}\leq0$ for all $x\geq1$.
Furthermore for $q\geq q_{*}(E,\beta)>\sqrt{(p-2)/p}$ we
see that $q\mapsto(1-q^{2})q^{p-2}$ is decreasing, so that this holds
for $q\geq q_{*}(E,\beta)-\eps$ for some $\eps$. sufficiently small.

It remains to consider the case of $\beta\geq\beta_{*}(E_{\infty})$, 
$E_{0}\leq E\leq E_{\infty}$, and $q_{*}(E_{\infty},\beta)\leq q\leq1$. 
If we let $q_{**}=q_{*}(E_{\infty},\beta)$, then  
\[
(1-q_{**}^{2})q_{**}^{p-2}=\frac{1}{\beta\sqrt{p(p-1)}}.
\]
Furthermore, since the function $q\mapsto(1-q^{2})q^{p-2}$ is decreasing,
we see the desired inequality holds for all $q\geq q_{**}$.
\end{proof}
We note the following immediate corollarly.
\begin{cor}\label{cor:ftap-is-frs}
For $p\geq 3$, let $E_0\leq E\leq E_\infty$, $\beta>\beta_*(E)$.
For all $q\geq q_*(E,\beta)$, we have that
\[
F_{TAP}(E,q,\beta)=F_{RS}(E,q,\beta).
\]
\end{cor}

\section{A Thouless--Anderson--Palmer decomposition}\label{sec:TAP-FE}

The goal of this section is to prove \prettyref{thm:ring-FE-lower-bound}.
This result will follow from a more general result which shows that decompositions of this 
type hold at an energy level provided a certain geometric condition holds on the collection 
of critical points. 
\begin{thm}\label{thm:abstract-nonsense}
Let $p\geq 3$, $0<r<1$, and $E_0<E< E_\infty$. If 
the landscape at $E$ is essentially $r$-separated, then for every $\sqrt{\frac{1+r}{2}}<q<1$,
there are sequences $\eps_N,\eta_N\to0$ and a sequence of random sets 
$A_N\subseteq \cC_N(E-\eps_N,E+\eps_N)$ such that for every $\beta>0$:
\begin{align*}
\frac{1}{N}\log\abs{A_N} &= \Theta(E)+o_\mathbb{P}(1)\\
F_N\left(\cup_{x\in A_N}B(x,q,\eta_{N});\beta\right) & = F_{TAP}(E,q,\beta)+\Theta(E)+o_{\mathbb{P}}(1)\\
\sup_{x\in A_N}\abs{F_N(B(x,q,\eta_N);\beta)-F_{TAP}(E,q,\beta)} &= o_{\mathbb{P}}(1),
\end{align*}
and such that the balls $\{B(x,q,\eta_{N})\}_{x\in A_N}$ are
pairwise disjoint and have their centres satisfy $R(x,y)<r$ with probability tending to 1.
\end{thm}

Before turning to the proof of this result let us pause to discuss how one verifies essential $r$-separation
and the proof of \prettyref{thm:ring-FE-lower-bound}.
To this end, let us recall the following results. The main result of 
\cite{subag2017complexity} states that
for any $E_{0}<E'\leq E_{\infty},$ we have  
\begin{equation}
\frac{\abs{\cC_N(-\infty,E)}}{\E\abs{\cC_N(-\infty,E)}}\to1\label{eq:subag-convergence-refined}
\end{equation}
 in probability. Recall also the following result of \cite{ABC13}: for any $E\in \R$,
 \begin{equation}\label{eq:ABC}
 \lim \frac{1}{N} \log \E \Theta_N(E) = \Theta(E).
 \end{equation}
 Note that \eqref{eq:subag-convergence-refined} and \eqref{eq:ABC} yield
 \begin{equation}\label{eq:ABC-subag-1}
 \Theta_N(E) \to \Theta(E)
 \end{equation}
 in probability for $E\in(E_0,E_\infty)$. On the other hand, by \cite[Corr. 13]{SubGibbs16}, we have that
for any $p\geq 3$ and any $r>0$, there is some
$\delta$ such that
\[
\lim P(\abs{\{x,x'\in\cC_N(E-\eps,E+\eps): x\neq \pm  x', \abs{R(x,x')}>r\}} \geq 1 )= 0
\]
for all $E$ and $\eps$ with $E_0\leq E-\eps <E+\eps< E_0+\delta$. 
Combing this with  \eqref{eq:ABC-subag-1} and the fact that $\Theta(E)>0$ if $E>E_0$ (so that for some $s>0$, we have $\Theta_N(E)>s$ with probability tending to $1$),
we immediately obtain:
\begin{lem}\label{lem:subag-separated}
Let $p\geq3$. For any $r>0$, the landscape at $E$ is essentially $r$-separated for all $E_0\leq E<E_0+\delta_0$
for some $\delta_0>0$.
\end{lem}
With this in hand, the proof of \prettyref{thm:ring-FE-lower-bound} is now complete.

\begin{proof}[\textbf{\emph{Proof of \prettyref{thm:ring-FE-lower-bound}}}]
This follows by combining \prettyref{lem:subag-separated} and \prettyref{thm:abstract-nonsense}.
\end{proof}

\subsection{Proof of \prettyref{thm:abstract-nonsense}}
Let us now turn to the proof of \prettyref{thm:abstract-nonsense}. 
We begin by noting that since $\E F_{2,N}(\cdot,\beta)$ is uniformly
Lipschitz by \prettyref{lem:uniform-lipschitz} and converges pointwise to $F_{2}(\cdot,\beta)$ and since $I(\cdot)$
is locally Lipschitz, it follows that for any $E<0$, $\beta>0$,
$0<q<1$, and $\eta>0$ with $\eta<\frac{1}{2}\left(q\wedge1-q\right)$,
\begin{equation}
\lim_{N\to\infty}\max_{t\in[q-\eta,q+\eta]}\{-\beta t^{p}E-I(t)+\E F_{2,N}(t,\beta)\}=\max_{t\in[q-\eta,q+\eta]}F_{TAP}(E,t,\beta).\label{eq:F_2-mean-converge}
\end{equation}
Consequently, by \prettyref{lem:uniform-lipschitz}, there is a $K>0$
such that for any $\eta<\frac{1}{2}q\wedge1-q$, $E_{0}\leq E\leq E_{\infty}$,
and $t\in[q-\eta,q+\eta]$, we have that
\begin{equation}
\abs{\lim_{N\to\infty}\max_{t\in[q-\eta,q+\eta]}\{-\beta t^{p}E-I(t)+\E F_{2,N}(t,\beta)\}-F_{TAP}(E,q,\beta)}\leq K\eta.\label{eq:F_2-mean-lb}
\end{equation}

The following key lemma shows that a macroscopic fraction of the (exponentially
many) critical points with a certain energy have free energies that
are well-approximated by $F_{TAP}$. This will follow by an application of the first moment 
method combined with a Kac--Rice-type argument. The precise form we use here
is from \cite[Lemma 14]{SubGibbs16} which is tailored exactly to our setting.
Recall from \cite{SubGibbs16} that restricted free energies satisfy
the \emph{tameness} property used there. 
In the following, let $\mathbf{n}=\sqrt{N}e_{1}$
and let $P_{u}$ denote the law of the Gaussian process, $(H(x)),$ conditioned
on $\mathbf{n}$ being a critical point with energy $Nu$.

For $E_{0}\leq E\leq E_{\infty}$ and $0<q<1$ and $\delta,\epsilon,K>0$
define the set 
\begin{align*}
A(E,q,\epsilon,\eta,\delta,K,\beta) & =\{x\in\cC_N(E-\epsilon,E+\epsilon):|F_{N}(B(x,q,\eta);\beta)-F_{TAP}(E,q,\beta)|<\delta+K(\eps+\eta)\}.
\end{align*}
We then have the following.

\begin{lem}
\label{lem:no-bad-bands}
For any $\beta>0$, $E\in(E_{0},E_{\infty})$, $n\geq1$ and sequence $0<q_1<\ldots<q_n<1$,
there are some $K,c,\delta_{0},\eta_0,\eps_0>0$ such that for any $\delta<\delta_{0}$,
 $\eta<\eta_0$,  $\eps<\eps_0$, and any sequence $0<t_N<1$ we have that
that
\[
\limsup\frac{1}{N}\log P\Big(\big|\bigcup_i (A(E,q_i,\epsilon,\eta,\delta,K,\beta)^{c})\big|\geq t_N \E e^{N\Theta_{N}(E+\epsilon)}\Big)\leq-c\delta^{2}+\limsup\frac{1}{N}\log(1/t_N)
\]
\end{lem}

\begin{proof}
Let $G^{(i)}_{N}(x)=F_N(B(x,q_i,\eta);\beta)$ and let
\begin{align*}
\tilde{A}(q) & =\{y\in\R:|y-F_{TAP}(E,q,\beta)|\leq\delta+K(\epsilon+\eta)\}
\end{align*}
where we will choose $K$ momentarily. As $E\in (E_0,E_\infty)$,
we may choose $\eps>0$ so that an $\eps$-neighbourhood of $E$ lies in this set as well.

Combining \eqref{eq:band-concentration} with \eqref{eq:F_2-mean-lb},
we see that if we choose $K$ to be larger than $2(K'+\beta)$ with $K'$ from \eqref{eq:F_2-mean-lb},
then for $N$ sufficiently large we have that for any $E-\eps<E'<E+\eps$,
\[
\frac{1}{N}\log P_{E'}(G^{(i)}_{N}(x)\in\tilde{A}(q_i)^{c})
\leq\frac{1}{N}\log P_{E'}(\abs{G^{(i)}_N(x)-\max_{s\in[q_i-\eta,q_i+\eta]}-\beta Es^{p}-I(s)+\E F_{2,N}(s,\beta)}\geq\delta)
\leq -c\delta^2
\]
for $\delta$ small enough but order 1, where in the second inequality
we have used that $K,\delta>0$. By a union bound, we then see that 
\begin{equation}\label{eq:prob-exp-bound}
\frac{1}{N}\log P_{E'}(\cup_i \{G^{(i)}_{N}(x)\in\tilde{A}(q_i)^{c}\}) \leq -c\delta^2 + \frac{\log n}{N}.
\end{equation}

On the other hand, since $G^{(i)}_N(x)$ is tame, we see that
by \cite[Lemma 14]{SubGibbs16}, for each $i$,
\[
\limsup\frac{1}{N}\log\left[\E\abs{\left\{ x\in\cC_N(E-\eps,E+\eps):G^{(i)}_{N}(x)\in\tilde{A}(q_i)^{c}\right\} }\right]\leq\sup_{t\in(E-\eps,E+\eps)}\Theta(t)-c\delta^{2}\leq\Theta(E+\eps)-c\delta^{2},
\]
where the last inequality follows since $\Theta$ is monotone increasing,
so that by a union bound, 
\begin{equation}\label{eq:moment-exp-bound}
\limsup \frac{1}{N} \log\left[\E\abs{\left\{ x\in\cC_N(E-\eps,E+\eps): \exists i\in[n]: G^{(i)}_{N}(x)\in\tilde{A}(q_i)^{c}\right\} }\right]
\leq \Theta(E+\eps)-c\delta^2
\end{equation}

Now, recall that by \eqref{eq:ABC}, we have that
\[
\lim \frac{1}{N}\log \E e^{N \Theta_N(E+\eps)} = \Theta(E+\eps).
\]
As such, if we let $S$ denote the event to be
bounded, then by Markov's inequality and the above two bounds, we obtain 
\[
\limsup\frac{1}{N}\log P(S)\leq-c\delta^{2}+\limsup\frac{1}{N}\log\left(\frac{1}{t_N}\right),
\]
as desired.
\end{proof}

Let us now turn to the main result of this section.

\begin{proof}[\textbf{\emph{Proof of \prettyref{thm:abstract-nonsense}}}]
As $E\in(E_0,E_\infty)$ and $q>0$, by \prettyref{lem:no-bad-bands} with $n=1$, we have that there is a $c>0$ such
that for $\eps>0$ sufficiently small (but order $1$ in $N$), $\lambda>0$, and $0<t<1$, with probability $1-\frac{1}{t}e^{-cN\lambda^{2}}$,
there is a set $A_N\subseteq\cC_N(E-\epsilon,E+\epsilon)$ with 
\begin{equation}\label{eq:ftap-fluc}
\sup_{x\in A_N} |F_{N}(B(x,q,\eta);\beta)- F_{TAP}(E,q,\eta)|\leq \lambda + K(\eta+\epsilon)
\end{equation}
for some $K>0$ and with 
\begin{align*}
\abs{A_N} & \geq\abs{\cC_N(E-\epsilon,E+\epsilon)}-t\E e^{N\Theta_{N}(E+\epsilon)}
  \geq\abs{\cC_N(-\infty,E+\epsilon)}\left(\left(1-tW_{N}\right)-\frac{\abs{\cC_N(-\infty,E-\epsilon)}}{\abs{\cC_N(-\infty,E+\epsilon)}}\right)
\end{align*}
where $W_{N}=(\E e^{N\Theta_N(E+\eps)})/|\cC_N(-\infty,E+\eps)|\to1$ in probability by \eqref{eq:subag-convergence-refined}. 
On the other hand, by construction we have that
\[
\abs{A_N}\leq \abs{\cC_N(-\infty,E+\eps)}
\]

Similarly by \eqref{eq:subag-convergence-refined} combined with \eqref{eq:ABC} and the fact that $\Theta(E)$
is strictly increasing and differentiable for $E\in(E_{0},E_{\infty})$, we see that
with probability tending to $1$, 
\[
\frac{\abs{\cC_N(-\infty,E-\epsilon)}}{\abs{\cC_N(-\infty,E+\epsilon)}}<e^{-c\epsilon N}.
\]
for some $c=c(E)>0$ consequently we have that with probability tending
to $1$, 
\begin{equation}\label{eq:cardinality-A}
\Theta_N(-\infty,E+\eps)\geq \frac{1}{N}\log|A_N|\geq\Theta_{N}(-\infty,E+\epsilon)+\frac{1}{N}\log\left((1-tW_{N})-e^{-cN}\right).
\end{equation}
As the last term in the righthand side in this display is tending to $0$ in probability, and since the landscape
at $E$ is essentially $r$-separated, we have that with probability tending to 1,
\[
\frac{1}{N}\log\abs{\{x,x'\in\cC_N(E-\eps,E+\eps): x\neq \pm x', r<R(x,x')<1\}} < \Theta_N(-\infty,E+\eps)-\delta'
\]
for some $\delta'>0$. Thus the $A_N$ as chosen above is exponentially larger than the size of the collection
of pairs of critical points that have large overlap. Thus
 we may slightly modify $A_N$ such that the above cardinality bound, \eqref{eq:cardinality-A}, still holds up to a $o(1)$ correction
 with probability tending to $1$ but
such that all pairs of distinct points in $A_N$ have $R(x,y)<r$.
As $q>\sqrt{\frac{1+r}{2}}$, we then see that for $\eta$ sufficiently small, the bands $\{B(x,q,\eta)\}_{x\in A_N}$
are disjoint and satisfy the condition on their centres. It remains to check the statement regarding the free energies.

If we let $\bar{F}=F_{TAP}(E,q,\beta)$ then on this event, we have
that 
\begin{align*}
F_N(\cup_{x\in A_N} &B(x,q,\eta),\beta) = \frac{1}{N}\log\sum_{x\in A_N}e^{NF_{N}(B(x,q,\eta);\beta)} =\bar{F}+\frac{1}{N}\log\sum_{x\in A_N}e^{N(F_{N}(B(x,q,\eta);\beta)-\bar{F})}\\
&\geq\bar{F}+\frac{1}{N}\log\abs{A_N}-\lambda-K(\eta+\epsilon)\\
 & \geq\bar{F}+\Theta_{N}(E+\epsilon)+\frac{1}{N}\log(1-tW_{N}-e^{-cN})-\lambda-K(\eta+\epsilon)+o(1),
\end{align*}
and, by the same reasoning, the matching upper bound,
\begin{align*}
\frac{1}{N}\log\sum_{x\in A_N}e^{NF_{N}(B(x,q,\eta);\beta)} 
  \leq \bar{F}+\Theta_{N}(E+\epsilon)+\lambda+K(\eta+\epsilon),
\end{align*}

Recall now that $\Theta_N(E+\eps) = \Theta(E+\eps)+o_{\mathbb{P}}(1)$ by \eqref{eq:ABC-subag-1}.
As such, taking $t$ fixed and choosing $\lambda,\eta,\epsilon\to0$ sufficiently
slowly in $N$ by diagonalization, we see that 
\begin{align*}
F_N(\cup_{x\in A_N}B(x,q,\eta);\beta)&= F_{TAP}(E,q,\beta)+\Theta(E)+o_{\mathbb{P}}(1)\\
\frac{1}{N}\log\abs{A_N} &= \Theta(E)+o_\mathbb{P}(1),
\end{align*}
which, after recalling \eqref{eq:ftap-fluc}, yields the desired. 
\end{proof}

Recall the TAP free energy, $F_{TAP}$ and replica symmetric TAP free energy 
$F_{RS}$.
Observe that for $\beta>\beta_{*}(E)$ and $q\geq q_{*}$, \prettyref{cor:ftap-is-frs}
yields $F_{TAP}=F_{RS}$.
We now observe the following corollary. We note here that for each $E$, $2q_*(E,\beta_*(E))^2-1 = (p-4)/p$ and,  by an explicit calculation,
$q_*(E_\infty,\beta_{sh}) = \sqrt{(p-2)/(p-1)}$ so that $2q_*(E_\infty,\beta_{sh})^2 -1 = (p-3)/(p-1)$.
\begin{cor}
\label{cor:FE-ring-LB} Let $p\geq4$. For any $\beta>\beta_{BBM}$,
there are $\iota, r>0$ with $r<(p-4)/p+\iota$
and a $\delta$ such that for any $E\in(E_0,E_0+\delta)$,
we have $\beta>\beta_*(E)$ and $q_*(E,\beta)\geq \sqrt{(1+r)/2}$.
Furthermore, for such $E$, if $q\geq q_*(E,\beta)$, 
then there are sequences $\epsilon_{N},\eta_{N} \to0$ such that 
there is a sequence of (random) sets $A_N\subseteq\cC_N(E-\eps_N,E+\eps_N)$
with:
\begin{align*}
\frac{1}{N}\log\abs{A_N} &= \Theta(E)+o_\mathbb{P}(1)\\
\sup_{x\in A_N}\abs{F_N(B(x,q,\eta_N);\beta)-F_{RS}(E,q,\beta)} &= o_{\mathbb{P}}(1)\\
F_N\left(\cup_{x\in A_N}B(x,q,\eta_{N});\beta\right) & = F_{RS}(E,q,\beta)+\Theta(E)+o_{\mathbb{P}}(1)
\end{align*}
 such that with probability tending to 1 , $\{B(x,q,\eta_{N})\}_{x\in\cC_N(E-\epsilon_{N},E+\epsilon_{N})}$
are pairwise disjoint and their centres satisfy $\abs{R(x,y)}<r$
for all distinct $x,y\in\cC_N(E-\epsilon_{N},E+\epsilon_{N})$. 
Furthermore this $r(\beta)$ and $\delta(\beta)$ are non-decreasing in $\beta$.
For $p=3$ the same holds for $\beta>T_{sh}^{-1}$ and  
$r<(p-3)/(p-1)+\iota$.
\end{cor}

\begin{proof}
Let us first take the case $p\geq 4$. Since $\beta>\beta_{BBM}=\beta_*(E_0)$,
we have that $q_*(E_0,\beta)>\sqrt{(p-2)/p}$ by monotonicity. Let $r>0$
be such that $$2 q_*(E_0,\beta)^2 -1 > r.$$  By \eqref{eq:fixed-point},
$q_*$ and $\beta_*(E)$ are continuous in $E$. As such, there is some $\delta'>0$
such that for all $E_0< E<E_0+\delta'$, the above inequality holds for $q_*(E,\beta)$ and $\beta_*(E)$. 
Furthermore since $q_*(E,\beta)$ is increasing in $\beta$, we see that for $\beta\geq \beta_*(E_0+\delta)$,
the above still holds for all such $E$, so that our choices of $r$ and $\delta'$ are non-decreasing in $\beta$.

Taking $\delta<\delta'(\beta) \vee \delta_0(r(\beta))$,
where $\delta_0$ is as in \prettyref{thm:ring-FE-lower-bound} we may conclude 
that there exist sets as in that theorem with the stated properties from that theorem.
Finally, noting that since $\beta>\beta_{*}(E)$, we have 
\begin{equation}\label{eq:qstar-4-bound}
q_{*}>\mbox{argmax}(1-q^{2})q^{p-2}=\sqrt{\frac{p-2}{p}}\geq\frac{1}{\sqrt{2}}.
\end{equation}
The result then follows 
since $\beta>\beta_{*}(E)$ and $q\geq q_{*}$ implies that
 $F_{TAP}=F_{RS}$ by \eqref{lem:RS-for-large-q}. 

Suppose now that $p=3$. Then by a direct calculation, we see that $q_{**}(\beta_{sh})=\frac{1}{\sqrt{2}}$,
thus for $q>q_{**}$ the above discussion still applies for any $E\in (E_0,E+\delta)$, provided $\beta_*(E)\leq\beta_{sh}$.
This inequality holds by \eqref{eq:temp-relations}.
\end{proof}

\section{Existence of the shattering phase}

Now for the proof of \prettyref{thm:shattering}. Recall
$q_{**}$ from \prettyref{sec:Free-energy-of-rings}. Note that by
construction, $\beta\mapsto q_{**}(\beta)$ is increasing for $\beta\geq\beta_{*}(E_{\infty})$.

\begin{proof}[\textbf{\emph{Proof of \prettyref{thm:shattering} and \prettyref{thm:shattering-p=3}}}]
Suppose that $\beta_{sh}<\beta<\beta_{s}$. Let us now choose $r(\beta),\delta(\beta)$ as in \prettyref{cor:FE-ring-LB},
and for $E\in(E_0,E_0+\delta(\beta))$, construct $A$ as in that Corollary. 
Since $\frac{1}{N}\log |A_N| \geq \Theta(E)+o_{\mathbb{P}}(1)$ and $\Theta(E)>0$, we have item (1).
Furthermore, since $\beta>\beta_{sh}$ and $p\geq 3$, the bands from \prettyref{cor:FE-ring-LB} are pairwise disjoint and their centres are nearly orthogonal yielding item (2) by the statement of that corollary.
Furthermore, by the statement of that corollary, we have
\[
\sup_{x\in A_N}\abs{F_N(B(x,q,\eta_N);\beta)-F_{RS}(E,q,\beta)} =o_\prob(1).
\]
On the other hand, since  $\Theta(E)>0$ we have that $F(\beta)- F_{RS}(E,q,\beta) > c'>0$ so that 
since $F_N(\beta)\to F(\beta)$ by \eqref{eq:free-energy} and \eqref{eq:concentration-of-free-energies},
we have that 
\[
F_N(\beta) -  \sup_{x\in A_N} F_N(B(x,q,\eta_N);\beta) >c'/2 
\]
with probability tending to 1. Thus item (3) holds.

To check item (4), first note that by \prettyref{cor:FE-ring-LB}, we have 
that for all $E\in(E_0,E_0+\delta(\beta)$ and $q\geq q_*(E,\beta)$,
\begin{align*}
F_{N}(\beta) & \geq F_{N}\left(\cup_{x\in A_N}B(x,q,\eta_{N});\beta\right)\geq F_{RS}(E,q,\beta)+\Theta(E)+o_{\mathbb{P}}(1)\geq\frac{\beta^{2}}{2}+o_{\mathbb{P}}(1).
\end{align*}
Thus if we choose $(E(\beta),q(\beta))$ to be the optimal pair from \prettyref{thm:FRS=FTAP}, 
it suffices to check that $E\in (E_0,E_0+\delta(\beta))$ for $\beta$ close enough to $\beta_{s}$. This follows since $E(\beta)\to E_0$ and $\delta(\beta)$ is non-decreasing.
\end{proof}

\section{Metastability below the Barrat--Burioni--M\'ezard temperature}

We begin this section by noting the following. 
\begin{lem}
Suppose that $q_{*}$ satisfies the fixed point equation \eqref{eq:fixed-point}
for $\abs{E}<\abs{E_{c}}$, then 
\begin{equation}
\frac{\partial^{2}}{\partial q^{2}}F_{RS}(E,q_{*},\beta)<0.\label{eq:q-star-strict-max}
\end{equation}
\end{lem}

\begin{proof}
By an explicit calculation, we see that
\[
\frac{\partial^{2}}{\partial q^{2}}F_{RS}(E,q_{*},\beta)=\frac{2-p(1-q^{2})}{(1-q^{2})^{2}}\left(\left(\frac{-E-\sqrt{E^{2}-E_{\infty}^{2}}}{E_{\infty}}\right)^{2}-1\right).
\]
since $q_{*}>\sqrt{\frac{p-2}{p}}$ we have $(1-q^{2})<\frac{2}{p}$
so that the first term is positive. The second term is negative since
$|E|>|E_{\infty}|$. 
\end{proof}
We now note the following useful Lemma whose proof is identical to that of
\prettyref{thm:ring-FE-lower-bound}. 
\begin{lem}\label{lem:feb-crit}
For $E_{0}<E<E_{\infty}$ and $\beta>\beta_{*}(E)$ and for any $q_{1}<q_{*}<q_{2}$
sufficiently close to $q_{*}$, there is an $h>0$,  such that for any  $\epsilon,\eta>0$ sufficiently
small, with probability tending to $1$,
\begin{equation}
\frac{1}{N}\log\Big|\big\{ x\in\cC_N(E-\epsilon,E+\epsilon):F_{N}(B(x,q_{*},\eta);\beta)-F_{N}(B(x,q_{1},\eta);\beta)>h\big\}\Big\vert>c.\label{eq:feb-crit-pt}
\end{equation}
\end{lem}

\begin{proof}
Let $q_1<q_*<q_2$.
By \prettyref{lem:no-bad-bands} with $n=3$ and the sequence $0<q_1<q_*<q_2$,
we have that for $\lambda,\eps,\eta$ sufficiently small,
there is a $C,c,c'>0$ such that with probability $1-C\exp(-c\lambda^2 N)$
there is a sequence of sets $A_N\subseteq\cC_N(E-\epsilon,E+\epsilon)$
with $\frac{1}{N}\log\abs{A}>c'$ such that for all $x\in A$,
\begin{equation}\label{eq:frs-qstar}
\begin{aligned}
F_{N}(B(x,q_{*},\eta);\beta) & \geq F_{RS}(E,q_{*},\beta)-\lambda-K(\eta+\epsilon)\\
F_{N}(B(x,q_{1},\eta);\beta) & \leq F_{RS}(E,q_{i},\beta)+\lambda+K(\eta+\epsilon).
\end{aligned}
\end{equation}
Here we have used that \prettyref{lem:RS-for-large-q}, $F_{TAP}=F_{RS}$ by choosing
$q_i$ close enough to $q_*$.
 
On the other hand, we see that by \eqref{eq:q-star-strict-max}
\[
F_{RS}(E,q_{*},\beta)>F_{RS}(E,q_{1},\beta)\vee F_{RS}(E,q_{2},\beta)
\]
for $q_{1}<q_{*}<q_{2}$ sufficiently close to $q_{*}$. Taking $\epsilon,\lambda,\eta$
small enough then yields the desired.
\end{proof}

Before turning to the main result of this section, let us briefly recall here the 
concepts of free energy barriers and free energy wells.
For any function $f:\cS_N\to\R$ we can define the follow entropy function
\[
I_f(a;\eps)= -\log\pi_{\beta^{-1}}(f \in B_\eps(a)).
\]
We recall here from \cite{BAJag17} that a function $f:\cS_{N}\to\R$
is said to have an \emph{ $\epsilon-$free energy barrier} of height $h$
if there are some $a<b<c$ with $|a-b|,|b-c|>2\epsilon$ and such
that
\[
I_f(b;\eps) -I_f(a;\eps)-I_f(c;\eps)\geq h.
\]
We remind the reader here that if there is a $1$-Lipschitz function
with a free energy barrier of height $Nh$ then the spectral gap can
be shown to be exponentially small in $N$ by a bounding the so-called \emph{$(1,\epsilon)$-difficulty}
of the Gibbs measure. 
More precisely, we have the following bound which is a specific case of \cite[Theorem 2.7]{BAJag17}. 
Suppose that $f$ has an $\eps$-free energy barrier of height $N h > \log 4$
and that $f$ is uniformly $K$-Lipschitz then we have 
\begin{equation}\label{eq:feb-bound}
\lambda_1 \leq \left(\frac{K}{\eps}\right)^2\frac{e^{-Nh}}{1-4 e^{-N h}}.
\end{equation}

Let us also recall here from \cite{BGJ18b} that a function $f:\cS_{N}\to\R$
is said to have an \emph{$\epsilon-$free energy well} of height $h$ on $[a,c]$
if the following holds: there is some $b\in [a,c]$ and $0<\eta<\eps$ such that 
$B_\eps(a),B_\eta(b)$ and $B_\eps(c)$ are disjoint and 
\[
\min\{I_f(a;\eps),I_f(c;\eps)\}-I_f(b;\eta)\geq h.
\]
Recall that $\norm{\nabla H}_\infty \leq K\sqrt{N}$ with probability $1-C\exp(-cN)$ for some $C,c,K>0$ (see, e.g., \cite[Theorem 4.3]{BGJ18a}). As such, if 
$f$ is a $1$-Lipschitz function with an free energy well of height $N h$ on some set $[a,b]$, 
the  exit time of that domain, started from the Gibbs measure conditioned on $B= \{f \in [a,b]\}$ is exponentially large
with high $\pi$-probability. More precisely, we have the following bound which is a specific case of  \cite[Theorem 7.4]{BGJ18b}.
There are universal $C',c'>0$ such that the following holds. If $f$ is $1$-Lipschitz, smooth that and has no critical values in an open neighborhood of some set $[a,b]$, then if $f$ has an  $\eta$-free energy well of height $Nh$ on $[a,b]$, we have that 
if 
\begin{equation}\label{eq:eps-condition}
\eta \leq \sqrt{\frac{C' h}{1+\beta K}}
\end{equation}
(here $K$ is from the norm bound of $\norm{\nabla H}_\infty$ above) we have that the exit time of $B$, call it $\tau_{B^c}$
satisfies
\begin{equation}\label{eq:few-bound}
\int Q_x\left(\tau_{B^c} \leq T\right) \pi(dx\vert E)\leq c'\left(1+\eta^{-4}N h T\right)\exp(-Nh)
\end{equation}
with probability $1-C\exp(-cN)$ in the law of $H$. With this in hand we now can prove the desired result.

\begin{proof}[\textbf{\emph{Proof of \prettyref{thm:metastability-main}}}]
Fix $\beta>\beta_{BBM}$. Let $\iota$ be such that $q_*(E_0,\beta)>\sqrt{\frac{p-2}{p}}+\iota$.
By continuity we have that  $q_*(E,\beta)>\sqrt{\frac{p-2}{p}}+\iota$ for all $E$ sufficiently close to $E_0$.
In particular, we may choose any $E$ with $E<E_0+\delta(\iota)$, where $\delta$ is as in \prettyref{thm:ring-FE-lower-bound}.

Fix $\eps>0$ sufficiently small that $E+\eps<E_0+\delta$ and let $A_N\subseteq\cC_N(E-\eps,E+\eps)$ as in Lemma~\ref{lem:feb-crit}, where
here we have taken $q_1$ and $q_2$ to be equidistant from $q_*$ with $2\eta = \abs{q_1-q_*}$.
Let $x_{0}\in A_N$ 
and let $f(x)=R(x,x_{0})$. 
Evidently $f$ is 1-Lipschitz. Furthermore
we have that for some $h>0$
\[
F_{N}(\{q_{*}-\eta<f<q_{*}+\eta\})>F_{N}(\{q_{1}-\eta<f<q_{1}+\eta\})\vee F_{N}(\{q_{1}-\eta<f<q_{2}+\eta\})+h
\]
for $\eta>0$ sufficiently small by our choice of $A_N$. In particular, we may choose $\eta$ so that \eqref{eq:eps-condition} 
holds.
Consequently $f$ has an $\eta$-free energy well of height a least $N h$ on $[q_1,q_2]$.  By \eqref{eq:few-bound},
we then see that there is a $C>0$ such that with probability tending to 1, for any $0\leq \theta<1$ we have
\[
\int Q_x(\tau_{B(x,q,\eta)^c} \geq e^{N\theta h})\pi_{\beta^{-1}}(dx\vert B(x,q,\delta))\leq C(1+N h e^{N \theta h}\eta^{-4})\exp(-Nh)\leq \exp(-N(1-\theta)h/2)
\]
for $N$ sufficiently large. As this holds simultaneously for all $x_0\in A_N$, this yields the desired exit time bound.

Now by 
 \prettyref{thm:ring-FE-lower-bound} (and \prettyref{lem:RS-for-large-q} again)
we have that $F_{RS}(E,q_{*},\beta)+c<F(\beta)$ for  some $c>0$
since $\Theta(E)>0$. Consequently, by \eqref{eq:frs-qstar}, 
\[
\frac{1}{N}\log\pi_{\beta^{-1}}\left(\{q_{*}-\eta<f(x)<q_{*}+\eta\}\right)=F_{N}(\{q_{*}-\eta<f(x)<q_{*}+\eta\},\beta)-F_{N}(\beta)<-c
\]
uniformly over all $x_0\in A_N$ with probability tending to $1$ for some $c$.
This yields $c$-subdominance. Finally the required cardinality bound on $A$ follows from \eqref{eq:feb-crit-pt}.
That we can take $\eps_N\to0$ follows by diagonalization.

Let us now turn to the desired spectral gap bound. Take $A_N$ as before except now we will fix $\eps>0$
as we do not need it to decay. The preceding discussion then still applies modulo this decay. In particular, 
fix again $x_0\in A$ and $f(x)=R(x,x_0)$.
Since $f(x)$ is bounded we see that there must exist some small
(random) interval $(a,b)$ such that $\pi(\{f\in(a,b)\})>c'$ with
probability tending to 1 for some $c'>0$. In particular, by $c$-subdominance we can choose this interval
so that $q_{1}>b$ or $q_{2}<a$. In either case, we see that with
probability tending to $1$, $f$ has an $\eta$-free energy barrier
of height at least $N h$ in the sense that for some $\eta>0$ sufficiently
small we have that one of the following holds:
\[
\begin{aligned}
\log\pi(f\in B_{\eta}(a))+\log\pi(f\in B_{\eta}(q_{*}))-\log\pi(f\in B_{\eta}(q_{1}))&>Nh+O(1), \quad \text{ or,}\\
\log\pi(f\in B_{\eta}(b))+\log\pi(f\in B_{\eta}(q_{*}))-\log\pi(f\in B_{\eta}(q_{2}))&>Nh+O(1).
\end{aligned}
\]
Thus by \eqref{eq:feb-bound}, it follows that
\[
\frac{1}{N}\log\lambda_{1}(L)<-h+o(1),
\]
as desired.
\end{proof}

\section{Replica symmetric TAP formula \label{sec:Replica-symmetric-TAP}}
In the following, let $V:[E_{0},E_{\infty}]\times[0,1]\times\R\to\R$
be given by 
\[
V(E,q;\beta)=F_{RS}(E,q,\beta)+\Theta(E)
\]
 and recall $q_{*},q_{**}$ from \prettyref{sec:Free-energy-of-rings}. Let $q_s=q_*(E_0,\beta_s)$.
\begin{proof}[\textbf{\emph{Proof of \prettyref{thm:FRS=FTAP} and \prettyref{thm:FRS=FTAP-p=3}}}]
As above we will work only with the inverse temperature $\beta$.
 First recall that by definition of $\beta_{s}=T^{-1}_{s}$ we have that $\lim_{N\to\infty}F_{N}(\beta)=\frac{\beta^{2}}{2}$
for $\beta\leq\beta_{s}$. Since $F_N(\beta)\geq F_N(A;\beta)$ for any $A\subseteq \cS_N$, we then have the upper bound, 
$\beta^2/2\geq \max V$, by \prettyref{cor:FE-ring-LB}. It remains to prove the matching
lower bound. 

To this end, we begin by recalling from \cite{subag2021free} that 
\begin{equation}\label{eq:rs-eq-bdry}
F_{RS}(E_0,q_*(E_0,\beta_s),\beta_s) = \beta_s^2/2.
\end{equation}
It remains to consider the case $\beta<\beta_s$. 
 We aim to show that 
\[
G(\beta)=\max_{(E,q)\in[E_{0},E_{0}+\eps]\times[q_{s},q_{s}+\delta]}V(E,q;\beta),
\]
for some $\eps,\delta$ small, has $G'(\beta)=\beta$ for $\beta$ sufficiently close
to $\beta_{sh}$. Since $G(\beta_{s})=\beta_{s}^2/2$ by the preceding display, the result will then follow by integration. 

To this end, we begin by observing that $V$ has the following property.
\begin{lem}\label{lem:interior-point}
For $\beta_{sh}<\beta<\beta_{s}$ sufficiently close to $\beta_s$,
there is some $\epsilon,\delta$ such that the maximum of 
\[
\max_{(E,q)\in[E_{0},E_{0}+\epsilon]\times[q_{s},q_{s}+\delta]}V(E,q,\beta) \label{eq:interior-pt-vp}
\]
is uniquely attained in the interior and such that the map $\beta\mapsto(E(\beta),q(\beta))$ is continuous
with $E(\beta)\to E_0$ as $\beta\to\beta_s$.
\end{lem}

The proof of this result is postponed to the end of this section.
By an envelope theorem (see, e.g., \prettyref{lem:envelope} below), $G$ is absolutely continuous,
$G(\beta)=G(\beta_{s})-\int G'(\beta),$ and for almost every $\beta$,
\[
G'(\beta)=\partial_{\beta}V(E(\beta),q(\beta),\beta),
\]
where $E(\beta),q(\beta)$ are an optimal choice of $E$ and $\beta$
in \prettyref{eq:interior-pt-vp}. 

Since the optimum of \prettyref{eq:interior-pt-vp} is at the interior
we have that at this optimum $\partial_{q}V(E,q,\beta)=0.$ By an
explicit calculation we see that  for $q\in [0,1)$ and $E\leq E_\infty$,
\begin{equation}
\begin{aligned}\label{eq:dVE-dVq}
\frac{\partial V}{\partial q} & =-\frac{\beta^2 p(p-1)q}{(1-q^{2})}\left((1-q^{2})q^{p-2}-\frac{-E-\sqrt{E^{2}-E_{\infty}^{2}}}{2\beta(p-1)}\right)\left((1-q^{2})q^{p-2}-\frac{-E+\sqrt{E^{2}-E_{\infty}^{2}}}{2\beta(p-1)}\right)\\
\frac{\partial V}{\partial E} & =-\beta q^{p}-E+p\left(\frac{E+\sqrt{E^{2}-E_{\infty}^{2}}}{2(p-1)}\right). 
\end{aligned}
\end{equation}
(Here at $E=E_{\infty}$ we view $\partial_{E}V$ as a left derivative.)
Recalling \eqref{eq:fixed-point} and the definitions of $q_*$ and $q_{**}$,  we see that for any $\beta$, 
and $E<E_\infty$,
\[
(1-q_{**}^2)q_{**}^{p-2}<\frac{-E+\sqrt{E^{2}-E_{\infty}^{2}}}{2\beta(p-1)},
\]
so that for $q\geq q_{**}$, the function $\partial_q V$
this has a unique zero at $q=q_{*}(E,\beta)$. Thus $q_{*}$ is the
second coordinate. 

Furthermore, we must have that $\partial_{E}V(E,q_{*},\beta)=0$,
which yields the following relation for $E$:
\begin{align*}
0=\partial_{E}V(E,q_{*},\beta) & =-\beta q^{p}-E+p\left(\frac{E+\sqrt{E^{2}-E_{\infty}^{2}}}{2(p-1)}\right)
  =-\beta q^{p}-E-\beta p(1-q^{2})q^{p-2}.
\end{align*}
Where in the last line we used the fixed point equation \eqref{eq:fixed-point}
for $q_*$. Plugging this in to the above we get that at the optimal pair,
\[
G'(\beta)=\partial_\beta F_{RS}(E,q,\beta)=-q^{p}E+\beta-q^{p}\left(\beta q^{p}+\beta pq^{p-2}(1-q^{2})\right)=\beta
\]
 as desired. Finally note that we have used here that the solutions solve the stated fixed point equations.
 The desired properties of the solution map come from taking $(E,q)$ to be the solutions from the above lemma. 
\end{proof}
We have used here the following envelope theorem (in the case of constant $b$).
\begin{lem}
\label{lem:envelope} Let $b:\R_+\to[0,1]$ be a non-decreasing function,
$f:X\times[0,1]\times[c,d]\to\mathbb{R}$ be a function that is differentiable in its third coordinate,
$t$, with uniformly bounded derivative, and let $g(t)=\max_{X\times[b(t),1]}f(x,a,t)$. 
Then $g$ is differentiable almost everywhere on $[c,d]$ and $g'(t)=\partial_{t}f(x,a,t)$
for any optimal pair $(x(t),a(t))$ with $g(t)=f(x(t),a(t),t)$. 
\end{lem}
\begin{proof}
We have that for $h$ sufficiently small
\begin{align*}
|g(t+h)-g(t)| & \leq\max_{x\in X,a\in [b,1]}|f(x,a,t+h)-f(x,a,t)|\leq\int\max_{X\times[b,1]}|\partial_{t}f(x,a,s)|<Ch
\end{align*}
so that $g(h)$ is absolutely continuous. Here in the second inequality
we used that $a(t)\leq a(t+h)$ so that for any optimal $\bar{a}(t+h)$
for fixed $t+h$, we have $\bar{a}(t+h)\geq a(t+h)\geq a(t)$.  In the third inequaliy we used the uniform boundedness of the derivative. Thus
it is differentiable almost everywhere. The first order optimality
condition then yields that for each $t$ where $g$ is differentiable,
$g'(t)=\partial_{t}f(x,a,t)$.
\end{proof}

In the following let $q_s=(E_0,q_*(E_0,\beta_s))$.

\begin{proof}[\textbf{\emph{Proof of \prettyref{lem:interior-point}}}]
We begin by showing that the function $DV=\left(\partial_{E}V,\partial_{q}V\right)$
is such that there is a locally well-defined family of solutions $\beta\mapsto(E(\beta),q(\beta))$
to $DV(E,q,\beta)=0$, with $(E(\beta_s),q(\beta_*))=(E_0,q_s)$.
To this end, observe that we may smoothly extend $V$ to an neighborhood
of $\left(E_{0},q_{*}\right)$ in $\R\times[0,1]$. Note that by \eqref{eq:rs-eq-bdry}, \eqref{eq:dVE-dVq},
and \eqref{eq:fixed-point},
we have $DV(E_{0},q_{s},\beta_{s})=0.$ Now at any point of the form
$(E,q,\beta)=(E,q_{*}(E,\beta),\beta)$ we have, by differentiating \eqref{eq:dVE-dVq}, that 
\begin{align*}
\partial_{E}^{2}V  =-1+p\left(\frac{\sqrt{E^{2}-E_{\infty}^{2}}+E}{2(p-1)\sqrt{E^{2}-E_{\infty}^{2}}}\right)\qquad
\partial_{q}\partial_{E}V  =-\beta pq^{p-1},
\end{align*}
and that for $f(q)=(1-q^{2})q^{p-2}$, 
\begin{align*}
\partial_{q}^{2}V & =-\frac{\beta^{2}p(p-1)q}{1-q^{2}}\left(f(q)-\frac{-E+\sqrt{E^{2}-E_{\infty}^{2}}}{2\beta(p-1)}\right)\left((p-2)(1-q^{2})q^{p-3}-2q^{p-1}\right)\\
 & =-\frac{\beta^{2}p(p-1)q}{1-q^{2}}\left(\frac{-E-\sqrt{E^{2}-E_{\infty}^{2}}}{2\beta(p-1)}-\frac{-E+\sqrt{E^{2}-E_{\infty}^{2}}}{2\beta(p-1)}\right)\left(\frac{p-2}{p}-q^{2}\right)pq^{p-3}\\
 & =-\frac{\beta p^{2}q^{p-2}}{1-q^{2}}\cdot\sqrt{E^{2}-E_{\infty}^{2}}\cdot\left(q^{2}-\frac{p-2}{p}\right),
\end{align*}
where in the second line we used that $q=q_{*}(E,\beta)$ satisfies
\eqref{eq:fixed-point}.

To show that this mapping is invertible at $(E,q,\beta)=(E_{0},q_{s},\beta_{s})$,
we claim that the determinant of the Hessian, $\det D^{2}V$ is strictly
positive. Given this claim, we obtain the existence of this family
of solutions and, by continuity and the second derivative test, that
this is in fact a one parameter family of local maxima. 

To prove this claim, note that for any $(E,q,\beta)$ as above, using
again \eqref{eq:fixed-point},
\begin{align*}
\det D^{2}V  =\left[\sqrt{E^{2}-E_{\infty}^{2}}+\beta pf(q)\right]\left[\frac{\beta p^{2}q^{p-2}}{\left(1-q^{2}\right)}\cdot\left(q^{2}-\frac{p-2}{p}\right)\right]-\left(\beta pq^{p-1}\right)^{2}.
\end{align*}
As such it suffices to show that 
\[
\beta pf(q)\cdot\frac{\beta p^{2}q^{p-2}}{(1-q^{2})}\left(q^{2}-\frac{p-2}{p}\right)-\beta^{2}p^{2}q^{2p-2}>0.
\]
Grouping like terms and cancelling, we note that this holds, provided
\[
(p-1)q^{2}-(p-2)>0.
\]
This holds for $q=q_{*}(E,\beta)$ provided $\beta>\beta_{sh}$ since
$q_{*}$ is decreasing in $E$ and since $q_{**}(\beta_{sh})=q_{*}(E_{\infty},\beta_{sh})=\sqrt{\frac{p-2}{p-1}}$
by a direct calculation. Thus $\det D^{2}V>0$ for any such $(E,q,\beta)$
and in particular for $(E_{0},q_{s},\beta_{s})$ as desired. 

The desired result then follows provided that for $\beta<\beta_{s}$,
the family $(E(\beta),q(\beta))$ has energy satisfying $E(\beta)>E_{0}$.
To see this, note that since $DV$ is $C^{1}$ in this region we have
that this one parameter family has
\[
\left(E'(\beta),q'(\beta)\right)=-\left(D^{2}V\right)^{-1}\partial_{\beta}DV.
\]
As all of the entries of $D^{2}V(E_{0},q_{s},\beta_{s})$ are strictly
negative, and the determinant was positive, we see that the the entries
of $(D^{2}V)^{-1}$ are negative on the diagonal and positive on the
off-diagonal. Furthermore at this point, $\partial_{\beta}\partial_{E}V=-q_{s}^{2}<0$
and $\partial_{\beta}\partial_{q}V=\frac{2}{\beta}\partial_{q}V=0$.
Thus $\partial_{\beta}E<0$, as desired. 

In summary, we have shown that in a neighborhood of $(E_0,q_s)$,
$V(E,q,\beta)$ has a unique, locally smooth, one parameter family of solutions
to the first order optimality conditions, $(E(\beta),q(\beta))$.
These solutions are local maxima by the second derivative test, are interior points of
$S= [E_0,E_\infty]\times [q_{**}(\beta_{sh}),1]$, and converge to $(E_0,q_s)$ as $\beta\to\beta_s$.  Thus the desired
result follows by choosing a small enough neighborhood around $(E_0,q_s)$
and intersecting with the box $S$.
\end{proof}

\section{The case $p=3$}\label{sec:p=3}
In the preceding, we have stated our main results for the case $p\geq 4$. It is natural to ask
what happens in the case $p=3$. While many of the results in the above hold
un-changed, the case $p=3$ becomes an interesting boundary case in many of our arguments.
Before discussing why, let us briefly summarize what results still hold in our setting.
The proofs of these results are given in the preceding simultaneously with the cases $p\geq 4$.

First note that our first main result, the TAP decomposition, still holds.
\begin{thm}
\label{thm:ring-FE-lower-bound-p=3} 
Let $p=3$. Then the conclusions of \prettyref{thm:ring-FE-lower-bound} hold unchanged.
\end{thm}
We next note that  the Barrat--Burioni--M\'ezard lower bound still applies with the caveat that the temperature
must be below the shattering transition.
 \begin{cor}[Barrat--Burioni--M\'ezard lower bound]
For $p= 3$, and any $T>0$ we have that for $\beta = T^{-1}$,
$F(\beta)\geq U(\beta).$
In particular, for $T<T_{sh}$ we have
$F(\beta)\geq U(\beta).$
\end{cor}

Consequently, our main results regarding shattering still apply.

\begin{thm}
\label{thm:FRS=FTAP-p=3} For $p= 3$ here is a $T_{s}< T_{0}\leq T_{sh}$
such that for all $T\in[T_s,T_{0})$ we have that for $\beta=T^{-1}$,
$F(\beta)= U(\beta) =\beta^2/2.$
Furthermore for such $T$, the maximum in \eqref{eq:BBM-func} is achieved at a pair $(E,q)=(E(\beta),q(\beta))$
with $q=q_*(E,\beta)$ and $E_0<E\leq E_\infty$ which satisfies
$E = - \beta(q^p +p(1-q^2)q^{p-2}),$
and such that the map $\beta\mapsto (E(\beta),q(\beta))$ is continuous and has $E(\beta)\to E_0$ as $\beta\to\beta_{RS}$.\end{thm}

\begin{thm}
\label{thm:shattering-p=3} For $p=3$, there is an $T_{0}>0$
with $T_{s}\leq T_{0}<T_{sh}$ such that free energy landscape is
shattered with probability tending to 1 for all $T_{0}<T\leq T_{sh}$. 
\end{thm}

Evidently our main result regarding metastability still applies for $T<T_{sh}$. 
Let us now briefly comment on what changes in this case. 

The need for the caveat that $T<T_{sh}$ in the Barrat--Burioni--M\'ezard (BBM) lower bound
is for the following reason. For this bound to hold, we need that the bands from \prettyref{thm:ring-FE-lower-bound-p=3}
are disjoint. In the latter theorem this holds by assumption since  we take $q$ slightly larger than $1/\sqrt{2}$. In the case of the Barrati--Burioni--M\'ezard 
bound, however, we need to know that $q_{**}(\beta)>1/\sqrt{2}$. This is guaranteed in the case $p=3$ and $\beta>\beta_{sh}$ by a 
direct calculation. For a proof of this fact, see \prettyref{cor:FE-ring-LB} above from which this corollary follows. (In particular, note that, given this fact, the proof of the BBM lower bound is immediate so we omit it as in the case $p\geq 4$.)

\section{Conjecture~\ref{conj:shattering}, Hypotheses~\ref{hyp:1} and~\ref{hyp:2}, and essential R-separation}\label{sec:conj1}

In this section, we briefly discuss the relationship between Hypothesis 1 and the preceding results and,
in particular, Conjecture~\ref{conj:shattering}. Let us begin by observing the following.
\begin{thm}
Hypothesis~\ref{hyp:1} implies Conjecture~\ref{conj:shattering} for every $p\geq 3$.
\end{thm}
\begin{proof}
As Hypothesis~\ref{hyp:1} holds, we have that Theorem~\ref{thm:abstract-nonsense}
holds for all $E\in(E_0,E_\infty)$ and all $q \geq q_{**}(\beta) > q_{**}(\beta_{sh})$ for $\beta>\beta_{sh}$.
Consequently, we have that
$F(\beta) \geq F_{BBM}(\beta).$
On the other hand, we have the following whose proof is deferred to the appendix.
\begin{lem}\label{lem:shattering-free-energy-calculus}
For $\beta_{sh}<\beta<\beta_s$, we have that
\[
\max_{\substack{(E,q)\in\cE_T\times [q_{**}(\beta),1]}} V(E,q,\beta) = \beta^2/2
\]
and this maximum is achieved at an interior point with $(E,q)=(E,q_*(E,\beta))$.
\end{lem}
Thus $F(\beta)=F_{BBM}(\beta)$. 
The theorem is them immediate by combining Theorem~\ref{thm:abstract-nonsense} with this Lemma
as in the proof of \prettyref{thm:shattering}.
\end{proof}

Let us also notice the following which is an immediate consequence of \prettyref{thm:abstract-nonsense}
\begin{thm}
Let $p\geq 4$. If Hypothesis~\ref{hyp:1} holds then for any 
$E\in(E_0,E_\infty)$, any $\sqrt{\frac{p-2}{p-1}}<q<1$ and any $\beta>0$, there are sequences
$\epsilon_{N},\eta_{N}\to0$ and a sequence of (random) sets $A_N\subseteq\cC_N(E-\eps_N,E+\eps_N)$
with:
\begin{align*}
\frac{1}{N}\log\abs{A_N} &= \Theta(E)+o_\mathbb{P}(1)\\
F_N\left(\cup_{x\in A_N}B(x,q,\eta_{N});\beta\right) & = F_{TAP}(E,q,\beta)+\Theta(E)+o_{\mathbb{P}}(1)\\
\sup_{x\in A_N}\abs{F_N(B(x,q,\eta_N);\beta)-F_{TAP}(E,q,\beta)} &= o_{\mathbb{P}}(1),
\end{align*}
and such that the balls $\{B(x,q,\eta_{N})\}_{x\in A_N}$ are
pairwise disjoint and have their centres satisfy $\abs{R(x,y)}<r$ with probability tending to 1.
\end{thm}

In a related direction, it seems natural to expect the following for $p\geq 4$.
\begin{hyp}\label{hyp:2}
For every $E_0\leq E\leq E_\infty$, the landscape a level $E$ is essentially $r$-separated for some $r<(p-4)/p+\iota=2(q_*(E,\beta_*(E))^2-1+\iota$ and some $\iota>0$ sufficiently small 
\end{hyp}
Note that under this Hypothesis, 
we would have the following as an immediate consequence of \prettyref{thm:abstract-nonsense}.
\begin{thm}
Let $p\geq 4$. If that Hypothesis~\ref{hyp:2} holds, then for every
$E\in(E_0,E_\infty)$, any $\sqrt{\frac{1}{2}}<q<1$, and any $\beta>0$, there are sequences
$\epsilon_{N},\eta_{N}\to0$ and a sequence of (random) sets $A_N\subseteq\cC_N(E-\eps_N,E+\eps_N)$
with:
\begin{align*}
\frac{1}{N}\log\abs{A_N} &= \Theta(E)+o_\mathbb{P}(1)\\
F_N\left(\cup_{x\in A_N}B(x,q,\eta_{N});\beta\right) & = F_{TAP}(E,q,\beta)+\Theta(E)+o_{\mathbb{P}}(1)\\
\sup_{x\in A_N}\abs{F_N(B(x,q,\eta_N);\beta)-F_{TAP}(E,q,\beta)} &= o_{\mathbb{P}}(1),
\end{align*}
and such that the balls $\{B(x,q,\eta_{N})\}_{x\in A_N}$ are
pairwise disjoint and have their centres satisfy $\abs{R(x,y)}<r$ with probability tending to 1 for some $r<(p-4)/p+\iota$ and some $\iota>0$.
\end{thm}

\appendix

\section{Shattering transition is in replica symmetric phase\label{sec:teperatures-not-equal}}
For the conveinence of the reader we provide here a direct proof that $T_{s}<T_{sh}<T_{BBM}$.
We show first that $T_{sh}>T_{s}$. To show this, recall that by the
replica symmetry test from \eqref{eq:tal-test} it suffices to show
that at $\beta_{sh}=T_{sh}^{-1}$,
\[
f(t)=\beta_{sh}^{2}t^{p}+\log(1-t)+t\leq0
\]
for all $0\leq t\leq1$. To see this first note that trivially $f(0)=f'(0)=0$.
note furthermore that 
\[
f'(t)=\frac{\beta_{sh}^{2}pt^{p-1}(1-t)-1+(1-t)}{(1-t)}
\]
Plugging in the value of $\beta_{sh}^{2}$ we see that $f'(t)\leq0$
for $t\in[0,1]$ with equality at $t=(p-2)/(p-1)$. 
We also have that $T_{sh}<T_{BBM}$.
To see this note that by a direct calculation
$T_{sh}<\beta_*^{-1}(E_\infty)\leq \beta_*^{-1}(E_0)=T_{BBM}$.

\section{The level of RSB in the co-dimension 1 model}

In this work, a key role was played by the co-dimension 1 model,
as it was related to the free energy,  $F_{2}(q,\beta)$, corresponding to $H(x)$ on a fixed
latitude $q$ around a critical point. While we focused on the case where $q\sim q_{*}$
which we saw was replica symmetric, it is evident from that argument
that for $q$ small enough, the model is not replica symmetry by Talagrand's
test. As such it's natural to ask how complex the model could be in
the sense of replica symmetry breaking, e.g., could it be, say, full
replica symmetry breaking? It turns out that this model is in fact
well behaved.
\begin{lem}
For any temperature and any $q\in(0,1)$, the Gibbs measure corresponding
to $\tilde{H}_{q}$ is always at most one step replica symmetry breaking.
\end{lem}

\begin{proof}
By the rule of signs from \cite{JagTob16boundingRSB} it suffices
to check the number of sign changes of 
\[
f=3\xi'''(t)^{2}-2\xi''(t)\xi''''(t).
\]
where $\xi$ is that for $\tilde{H_{q}}$. By an explicit computation,
\[
f=(p-2)(p-1)^{2}p^{3}(1-q^{2})^{6}(q^{2}+(1-q^{2})t)^{2p-6}
\]
which is non-negative. Thus $f\geq0$ so that by the rule of signs,
the model is at most 1 RSB.
\end{proof}
\section{Proof of \prettyref{lem:shattering-free-energy-calculus}}
The proof of this result will follow from the following lemma.
\begin{lem}\label{lem:boundary-good}
We have the following.
\begin{enumerate}
\item For all $\beta>0$ and all $c>0$, there is some $\eps(\beta,c)>0$ such that 
\[
\inf_{E\in(E_0,E_\infty)}\inf_{1-\eps\leq q\leq1} \partial_q V <-c, 
\]
\item For all $\beta_{sh}<\beta$
\[
\begin{aligned}
\partial_q V(E,q_{**},\beta) >0 \quad \forall E_0\leq E\leq E_\infty\qquad\text{ and }\qquad
\partial_E V(E_\infty,q,\beta)<0 \quad \forall q>q_{**}
\end{aligned}
\]
\item For $\beta_*(E_\infty)<\beta$, $E\in [E_0,E_\infty]$ and $q_{**}(\beta)\leq q\leq 1$ we have that
\[
\begin{aligned}
\partial_q V >0 \quad \forall q<q_*,\text{ and }\qquad 
\partial_q V <0  q>q_* \quad\forall q>q_*,\text{ and }\qquad \partial_q V = 0 & q= q_*.
\end{aligned}
\]
In particular, $\max_q V(E,q,\beta)$ is uniquely attained at $q_*$.
\item We have that $\partial_E V(E_0,q_*(E_0,\beta),\beta)\geq 0$ for $\beta\leq \beta_s$ with equality
if and only if $\beta=\beta_s$.
\end{enumerate}
\end{lem}
Indeed, given this claim we see that the maximum of $V$ occurs on the interior. From here the proof is identical to that of
\prettyref{thm:FRS=FTAP} with this fact in place of \prettyref{lem:interior-point} after noting that $\beta\mapsto q_{**}(\beta)$ 
is strictly increasing so that \prettyref{lem:envelope} applies.

It remains to prove the claim. 
\begin{proof}[Proof of \prettyref{lem:boundary-good}]
\textbf{Item 1.} Observing that  as $q\to1$, $f(q)=(1-q^2)q^{p-2}$ has
\[
(f(q)-\frac{-E-\sqrt{E^2-E_\infty^2}}{2\beta (p-1)})(f(q) -\frac{-E+\sqrt{E^2-E_\infty^2}}{2\beta (p-1)})\to c(\beta)>0,
\]
so that $\partial_q V(E,q,\beta)\to -\infty$ for each $E$. The proof then follows by continuity of $\partial_qV$ away from $q=1$ and a compactness argument.

\textbf{Item 2.} This follows by a direct calculation: 
\begin{align*}
\partial_{q}V(E,q_{**}) & =-\frac{\beta^2p(p-1)q_{**}}{1-q_{**}^{2}}\left(f(q_{**})-\frac{-E-\sqrt{E^{2}-E_{\infty}^{2}}}{2\beta(p-1)}\right)\left(f(q_{**})-\frac{-E+\sqrt{E^{2}-E_{\infty}^{2}}}{2\beta(p-1)}\right)\\
 & =-\frac{\beta^2p(p-1)q_{**}}{1-q_{**}^{2}}\left(\frac{E-E_{\infty}+\sqrt{E^{2}-E_{\infty}^{2}}}{2\beta(p-1)}\right)\left(\frac{E-E_{\infty}-\sqrt{E^{2}-E_{\infty}^{2}}}{2\beta(p-1)}\right)>0
\end{align*}
since the third term is obviously negative and the second term is
positive since $x+1+\sqrt{x^{2}-1}>0\quad\forall x<-1$ . Also, for
$q\geq q_{**}$ and for $\beta>\beta_{sh}$ since the map $\beta\mapsto q_{**}(\beta)$
is strictly increasing, we have $q_{**}(\beta) >q_{**}(\beta_{sh})$, so that
\begin{align*}
\partial_{E}V(E_{\infty},q) & =-\beta q^{p}-E_{\infty}+p\left(\frac{E_{\infty}}{2(p-1)}\right)
  <-\beta q_{**}^{p}(\beta_{sh})-E_{\infty}+p\left(\frac{E_{\infty}}{2(p-1)}\right)=0.
\end{align*}

\textbf{Item 3.} For $\beta>\beta_*(E_\infty)\geq \beta_*(E)$, we have that $q_*$ is well-defined.
Furthermore, by \eqref{eq:dVE-dVq}, we have $\partial_q V$ is the product of three terms. The first 
is clearly negative.  For the third, note that if $q \geq q_{**}>\sqrt{(p-2)/p}$, we have $f(q)\leq f(q_{**})$
so that for $E<E_\infty$, we have
\[
f(q)-\frac{-E +\sqrt{E^2-E_\infty^2}}{2\beta(p-1)} \leq f(q_{**})-\frac{-E +\sqrt{E^2-E_\infty^2}}{2\beta(p-1)}
=\frac{(E-E_\infty)-\sqrt{E^2-E_\infty^2}}{2\beta(p-1)}<0.
\]
this $\partial_q V$ has the same sign as he second term in \eqref{eq:dVE-dVq}. That term is zero at $q_*=0$
and  positive or negative as $q<q_*$ or $q>q_*$ respectively as $f$ is decreasing for $q>\sqrt{(p-2)/p}$.

\textbf{Item 4.} Recall that by \eqref{eq:rs-eq-bdry} we have  $\partial_E V(E_0,q_s,\beta_s)=0$
for $q_s=q_*(E_0,\beta_s)$. Since $\beta \to q_*(E_0,\beta)$ is strictly increasing we have that 
$\beta \mapsto \beta q_*^p$ is as well so that 
\[
\partial_E V(E_0,q_*(E_0,\beta),\beta)=-\beta q^p +g(E)>\partial_E V(E_0,q_s,\beta_s)=0,
\]
for $\beta<\beta_s$ where $g$ is a quantity that does not depend on $\beta$.
\end{proof} 

\bibliographystyle{plain}
\bibliography{spinglasses}

\begin{thebibliography}{10}

\bibitem{achlioptas2008algorithmic}
Dimitris Achlioptas and Amin Coja-Oghlan.
\newblock Algorithmic barriers from phase transitions.
\newblock In {\em 2008 49th Annual IEEE Symposium on Foundations of Computer
  Science}, pages 793--802. IEEE, 2008.

\bibitem{ACRT11}
Dimitris Achlioptas, Amin Coja-Oghlan, and Federico Ricci-Tersenghi.
\newblock On the solution-space geometry of random constraint satisfaction
  problems.
\newblock {\em Random Struct. Algorithms}, 38(3):251--268, May 2011.

\bibitem{aizenman2003extended}
Michael Aizenman, Robert Sims, and Shannon~L Starr.
\newblock Extended variational principle for the sherrington-kirkpatrick
  spin-glass model.
\newblock {\em Physical Review B}, 68(21):214403, 2003.

\bibitem{alaoui2020optimization}
Ahmed~El Alaoui, Andrea Montanari, and Mark Sellke.
\newblock Optimization of mean-field spin glasses.
\newblock {\em arXiv preprint arXiv:2001.00904}, 2020.

\bibitem{Arg08}
Louis-Pierre Arguin.
\newblock A remark on the infinite-volume gibbs measures of spin glasses.
\newblock {\em Journal of Mathematical Physics}, 49(12):125204, 2008.

\bibitem{ABA13}
Antonio Auffinger and G\'erard Ben~Arous.
\newblock Complexity of random smooth functions on the high-dimensional sphere.
\newblock {\em Ann. Probab.}, 41(6):4214--4247, 2013.

\bibitem{ABC13}
Antonio Auffinger, G{\'e}rard Ben~Arous, and Ji{\v{r}}{\'{\i}} {\v{C}}ern{\'y}.
\newblock Random matrices and complexity of spin glasses.
\newblock {\em Comm. Pure Appl. Math.}, 66(2):165--201, 2013.

\bibitem{AuffChen13}
Antonio Auffinger and Wei-Kuo Chen.
\newblock On properties of {P}arisi measures.
\newblock {\em Probab. Theory Related Fields}, 161(3-4):817--850, 2015.

\bibitem{auffinger2020number}
Antonio Auffinger and Julian Gold.
\newblock The number of saddles of the spherical $ p $-spin model.
\newblock {\em arXiv preprint arXiv:2007.09269}, 2020.

\bibitem{AufJag16}
Antonio Auffinger and Aukosh Jagannath.
\newblock Thouless-{A}nderson-{P}almer equations for generic {$p$}-spin
  glasses.
\newblock {\em Ann. Probab.}, 47(4):2230--2256, 2019.

\bibitem{baity2018activated}
Marco Baity-Jesi, Alexandre Achard-de Lustrac, and Giulio Biroli.
\newblock Activated dynamics: An intermediate model between the random energy
  model and the p-spin model.
\newblock {\em Physical Review E}, 98(1):012133, 2018.

\bibitem{barrat1996dynamics}
Alain Barrat, Raffaella Burioni, and Marc M{\'e}zard.
\newblock Dynamics within metastable states in a mean-field spin glass.
\newblock {\em Journal of Physics A: Mathematical and General}, 29(5):L81,
  1996.

\bibitem{bauerschmidt2017very}
Roland Bauerschmidt and Thierry Bodineau.
\newblock A very simple proof of the {LSI} for high temperature spin systems.
\newblock {\em J. Funct. Anal.}, 276(8):2582--2588, 2019.

\bibitem{belius2019tap}
David Belius and Nicola Kistler.
\newblock The tap--plefka variational principle for the spherical sk model.
\newblock {\em Communications in Mathematical Physics}, 367(3):991--1017, 2019.

\bibitem{BABC08}
G\'erard Ben~Arous, Anton Bovier, and Ji{\v{r}}{\'{\i}} {\v{C}}ern{\'y}.
\newblock Universality of the {REM} for dynamics of mean-field spin glasses.
\newblock {\em Comm. Math. Phys.}, 282(3):663--695, 2008.

\bibitem{BABG02}
G{\'e}rard Ben~Arous, Anton Bovier, and V{\'e}ronique Gayrard.
\newblock Aging in the random energy model.
\newblock {\em Physical review letters}, 88(8):087201, 2002.

\bibitem{BADG01}
G{\'e}rard Ben~Arous, Amir Dembo, and Alice Guionnet.
\newblock Aging of spherical spin glasses.
\newblock {\em Probab. Theory Related Fields}, 120(1):1--67, 2001.

\bibitem{BADG06}
G{\'e}rard Ben~Arous, Amir Dembo, and Alice Guionnet.
\newblock Cugliandolo-{K}urchan equations for dynamics of spin-glasses.
\newblock {\em Probab. Theory Related Fields}, 136(4):619--660, 2006.

\bibitem{BGJ18b}
G\'{e}rard Ben~Arous, Reza Gheissari, and Aukosh Jagannath.
\newblock Algorithmic thresholds for tensor {PCA}.
\newblock {\em Ann. Probab.}, 48(4):2052--2087, 2020.

\bibitem{BGJ18a}
G\'{e}rard Ben~Arous, Reza Gheissari, and Aukosh Jagannath.
\newblock Bounding flows for spherical spin glass dynamics.
\newblock {\em Comm. Math. Phys.}, 373(3):1011--1048, 2020.

\bibitem{BAJag17}
G{\'e}rard Ben~Arous and Aukosh Jagannath.
\newblock Spectral gap estimates in mean field spin glasses.
\newblock {\em Comm. Math. Phys.}, 361(1):1--52, 2018.

\bibitem{arous2020geometry}
G{\'e}rard Ben~Arous, Eliran Subag, and Ofer Zeitouni.
\newblock Geometry and temperature chaos in mixed spherical spin glasses at low
  temperature: the perturbative regime.
\newblock {\em Communications on Pure and Applied Mathematics},
  73(8):1732--1828, 2020.

\bibitem{Bolt14}
Erwin Bolthausen.
\newblock An iterative construction of solutions of the {TAP} equations for the
  {S}herrington-{K}irkpatrick model.
\newblock {\em Comm. Math. Phys.}, 325(1):333--366, 2014.

\bibitem{BoltSznit98}
Erwin Bolthausen and Alain-Sol Sznitman.
\newblock On {R}uelle's probability cascades and an abstract cavity method.
\newblock {\em Comm. Math. Phys.}, 197(2):247--276, 1998.

\bibitem{BCKM98}
Jean-Philippe Bouchaud, Leticia~F Cugliandolo, Jorge Kurchan, and Marc
  M{\'e}zard.
\newblock Out of equilibrium dynamics in spin-glasses and other glassy systems.
\newblock {\em Spin glasses and random fields}, pages 161--223, 1998.

\bibitem{BovFag05}
Anton Bovier and Alessandra Faggionato.
\newblock Spectral characterization of aging: the {REM}-like trap model.
\newblock {\em Ann. Appl. Probab.}, 15(3):1997--2037, 2005.

\bibitem{CastCav05}
Tommaso Castellani and Andrea Cavagna.
\newblock Spin-glass theory for pedestrians.
\newblock {\em Journal of Statistical Mechanics: Theory and Experiment},
  2005(05):P05012, 2005.

\bibitem{CerWas17}
Ji{\v{r}}{\'{\i}} {\v{C}}ern{\'y} and Tobias Wassmer.
\newblock Aging of the {M}etropolis dynamics on the random energy model.
\newblock {\em Probab. Theory Related Fields}, 167(1-2):253--303, 2017.

\bibitem{ChenSph13}
Wei-Kuo Chen.
\newblock The {A}izenman-{S}ims-{S}tarr scheme and {P}arisi formula for mixed
  {$p$}-spin spherical models.
\newblock {\em Electron. J. Probab.}, 18:no. 94, 14, 2013.

\bibitem{chen2018generalized}
Wei-Kuo Chen, Dmitry Panchenko, and Eliran Subag.
\newblock The generalized tap free energy.
\newblock {\em arXiv preprint arXiv:1812.05066}, 2018.

\bibitem{COP19}
Amin Coja-Oghlan and Will Perkins.
\newblock Bethe states of random factor graphs.
\newblock {\em Comm. Math. Phys.}, 366(1):173--201, 2019.

\bibitem{CriSom92}
Andrea Crisanti and Hans~J{\"u}rgen Sommers.
\newblock The spherical $p$-spin interaction spin glass model: the statics.
\newblock {\em Zeitschrift f{\"u}r Physik B Condensed Matter}, 87(3):341--354,
  1992.

\bibitem{CugKur93}
Leticia~F. Cugliandolo and Jorge Kurchan.
\newblock Analytical solution of the off-equilibrium dynamics of a long-range
  spin-glass model.
\newblock {\em Phys. Rev. Lett.}, 71:173--176, Jul 1993.

\bibitem{cugliandolo1995weak}
LF~Cugliandolo and J~Kurchan.
\newblock Weak ergodicity breaking in mean-field spin-glass models.
\newblock {\em Philosophical Magazine B}, 71(4):501--514, 1995.

\bibitem{DGM07}
Amir Dembo, Alice Guionnet, and Christian Mazza.
\newblock Limiting dynamics for spherical models of spin glasses at high
  temperature.
\newblock {\em J. Stat. Phys.}, 128(4):847--881, 2007.

\bibitem{dembo2013factor}
Amir Dembo, Andrea Montanari, and Nike Sun.
\newblock Factor models on locally tree-like graphs.
\newblock {\em Annals of Probability}, 41(6):4162--4213, 2013.

\bibitem{DS20}
Amir Dembo and Eliran Subag.
\newblock Dynamics for spherical spin glasses: disorder dependent initial
  conditions.
\newblock {\em J. Stat. Phys.}, 181(2):465--514, 2020.

\bibitem{DSS15}
Jian Ding, Allan Sly, and Nike Sun.
\newblock Proof of the satisfiability conjecture for large k.
\newblock In {\em Proceedings of the Forty-Seventh Annual ACM on Symposium on
  Theory of Computing}, STOC '15, pages 59--68, New York, NY, USA, 2015. ACM.

\bibitem{DSS16}
Jian Ding, Allan Sly, and Nike Sun.
\newblock Maximum independent sets on random regular graphs.
\newblock {\em Acta Math.}, 217(2):263--340, 2016.

\bibitem{eldan2020spectral}
Ronen Eldan, Frederic Koehler, and Ofer Zeitouni.
\newblock A spectral condition for spectral gap: Fast mixing in
  high-temperature ising models.
\newblock {\em arXiv preprint arXiv:2007.08200}, 2020.

\bibitem{FraPar95}
Silvio Franz and Giorgio Parisi.
\newblock Recipes for metastable states in spin glasses.
\newblock {\em Journal de Physique I}, 5(11):1401--1415, 1995.

\bibitem{gamarnik2014limits}
David Gamarnik and Madhu Sudan.
\newblock Limits of local algorithms over sparse random graphs.
\newblock In {\em Proceedings of the 5th conference on Innovations in
  theoretical computer science}, pages 369--376. ACM, 2014.

\bibitem{gayrard2019aging}
V{\'e}ronique Gayrard.
\newblock Aging in metropolis dynamics of the rem: a proof.
\newblock {\em Probability Theory and Related Fields}, 174(1):501--551, 2019.

\bibitem{GJ16}
Reza Gheissari and Aukosh Jagannath.
\newblock On the spectral gap of spherical spin glass dynamics.
\newblock {\em Ann. Inst. Henri Poincar\'{e} Probab. Stat.}, 55(2):756--776,
  2019.

\bibitem{Guerra2003}
Francesco Guerra.
\newblock Broken replica symmetry bounds in the mean field spin glass model.
\newblock {\em Communications in Mathematical Physics}, 233(1):1--12, Feb 2003.

\bibitem{holley1987logarithmic}
Richard Holley and Daniel Stroock.
\newblock Logarithmic sobolev inequalities and stochastic ising models.
\newblock {\em Journal of Statistical Physics}, 46(5-6):1159--1194, 1987.

\bibitem{jagannath2017approximate}
Aukosh Jagannath.
\newblock Approximate ultrametricity for random measures and applications to
  spin glasses.
\newblock {\em Comm. Pure Appl. Math}, 70(4):611--664, 2017.

\bibitem{jagannath2019dynamics}
Aukosh Jagannath.
\newblock Dynamics of mean field spin glasses on short and long timescales.
\newblock {\em Journal of Mathematical Physics}, 60(8):083305, 2019.

\bibitem{JagTob16boundingRSB}
Aukosh Jagannath and Ian Tobasco.
\newblock Bounds on the complexity of replica symmetry breaking for spherical
  spin glasses.
\newblock {\em Proceedings of the American Mathematical Society},
  146(7):3127--3142, 2018.

\bibitem{KT87}
T.~R. Kirkpatrick and D.~Thirumalai.
\newblock p-spin-interaction spin-glass models: Connections with the structural
  glass problem.
\newblock {\em Phys. Rev. B}, 36:5388--5397, Oct 1987.

\bibitem{ko2020free}
Justin Ko.
\newblock Free energy of multiple systems of spherical spin glasses with
  constrained overlaps.
\newblock {\em Electronic Journal of Probability}, 25, 2020.

\bibitem{KMRTSZ07}
Florent Krzaka{\l}a, Andrea Montanari, Federico Ricci-Tersenghi, Guilhem
  Semerjian, and Lenka Zdeborov{\'a}.
\newblock Gibbs states and the set of solutions of random constraint
  satisfaction problems.
\newblock {\em Proceedings of the National Academy of Sciences},
  104(25):10318--10323, 2007.

\bibitem{KurParVir93}
Jorge Kurchan, Giorgio Parisi, and Miguel~Angel Virasoro.
\newblock Barriers and metastable states as saddle points in the replica
  approach.
\newblock {\em Journal de Physique I}, 3(8):1819--1838, 1993.

\bibitem{Mat00}
Pierre Mathieu.
\newblock Convergence to equilibrium for spin glasses.
\newblock {\em Comm. Math. Phys.}, 215(1):57--68, 2000.

\bibitem{MatMou15}
Pierre Mathieu and Jean-Christophe Mourrat.
\newblock Aging of asymmetric dynamics on the random energy model.
\newblock {\em Probab. Theory Related Fields}, 161(1-2):351--427, 2015.

\bibitem{MM09}
Marc M{\'e}zard and Andrea Montanari.
\newblock {\em Information, physics, and computation}.
\newblock Oxford Graduate Texts. Oxford University Press, Oxford, 2009.

\bibitem{MPV87}
Marc M{\'e}zard, Giorgio Parisi, and Miguel~Angel Virasoro.
\newblock {\em Spin glass theory and beyond}, volume~9.
\newblock World scientific Singapore, 1987.

\bibitem{MPZ02}
Marc M\'{e}zard, Giorgio Parisi, and Riccardo Zecchina.
\newblock Analytic and algorithmic solution of random satisfiability problems.
\newblock {\em Science}, 297(5582):812--815, 2002.

\bibitem{montanari2021optimization}
Andrea Montanari.
\newblock Optimization of the sherrington--kirkpatrick hamiltonian.
\newblock {\em SIAM Journal on Computing}, (0):FOCS19--1, 2021.

\bibitem{PanchUlt13}
Dmitry Panchenko.
\newblock The {P}arisi ultrametricity conjecture.
\newblock {\em Ann. of Math. (2)}, 177(1):383--393, 2013.

\bibitem{PanchSKBook}
Dmitry Panchenko.
\newblock {\em The Sherrington-Kirkpatrick model}.
\newblock Springer, 2013.

\bibitem{PanchPF14}
Dmitry Panchenko.
\newblock The {P}arisi formula for mixed {$p$}-spin models.
\newblock {\em Ann. Probab.}, 42(3):946--958, 2014.

\bibitem{rahman2017local}
Mustazee Rahman and Balint Virag.
\newblock Local algorithms for independent sets are half-optimal.
\newblock {\em The Annals of Probability}, 45(3):1543--1577, 2017.

\bibitem{sly2016reconstruction}
Allan Sly and Yumeng Zhang.
\newblock Reconstruction of colourings without freezing.
\newblock {\em arXiv preprint arXiv:1610.02770}, 2016.

\bibitem{subag2017complexity}
Eliran Subag.
\newblock The complexity of spherical $ p $-spin models---a second moment
  approach.
\newblock {\em The Annals of Probability}, 45(5):3385--3450, 2017.

\bibitem{SubGibbs16}
Eliran Subag.
\newblock The geometry of the gibbs measure of pure spherical spin glasses.
\newblock {\em Inventiones mathematicae}, pages 1--75, 2017.

\bibitem{subag2018free}
Eliran Subag.
\newblock Free energy landscapes in spherical spin glasses.
\newblock {\em arXiv preprint arXiv:1804.10576}, 2018.

\bibitem{subag2021free}
Eliran Subag.
\newblock The free energy of spherical pure $ p $-spin models--computation from
  the tap approach.
\newblock {\em arXiv preprint arXiv:2101.04352}, 2021.

\bibitem{TALAGRAND2003477}
Michel Talagrand.
\newblock On guerra's broken replica-symmetry bound.
\newblock {\em Comptes Rendus Mathematique}, 337(7):477 -- 480, 2003.

\bibitem{TalSphPF06}
Michel Talagrand.
\newblock Free energy of the spherical mean field model.
\newblock {\em Probab. Theory Related Fields}, 134(3):339--382, 2006.

\end{thebibliography}

\end{document}